\documentclass[a4paper,10pt]{article}
\usepackage{styles}
\usepackage{braket}
\usepackage{soul}
\usepackage[toc,page]{appendix}

\usepackage{authblk}

\renewcommand{\R}{\mathbb R}
\renewcommand{\N}{\mathbb N}
\renewcommand{\E}{\mathbb E}
\renewcommand{\C}{\mathbb C}

\renewcommand{\Prob}{\mathbb P}
\renewcommand{\eqL}{\stackrel{\mathcal{L}}{=}}
\renewcommand{\levy}{L\'evy }
\renewcommand{\cramer}{Cram\'er }
\renewcommand{\cramers}{Cram\'er's }

\renewcommand{\cadlag}{c\`adl\`ag }

\renewcommand{\notzero}{\backslash\{0\}}

\title{On the finiteness and tails of perpetuities under a Lamperti-Kiu MAP}
\author{Larbi Alili\footnote{The University of Warwick, CV47AL, Coventry, UK. \textit{E-mail}: \textbf{l.alili@warwick.ac.uk}} \:and David Woodford\footnote{The University of Warwick, CV47AL, Coventry, UK. \textit{E-mail}: \textbf{david.l.woodford@bath.edu}}\footnote{This work was supported by EPSRC as part of the MASDOC DTC, Grant No. EP/H023364/1.}}
\date{\today}

\begin{document}

\maketitle

\begin{abstract}
  Consider a Lamperti-Kiu Markov additive process $(J, \xi)$  on $\{+, -\}\times\R\cup \{-\infty\}$, where $J$ is the modulating Markov chain component.   First, we study the finiteness of the exponential functional and then consider its moments and tail asymptotics under \cramers condition. In the strong subexponential case we determine the subexponential tails of the exponential functional under some further assumptions.
\end{abstract}


\subject{ Primary 60G18; 60J45; Secondary 91B70; 91B30 }

\section{Introduction and Preliminaries}
    Let $E:=\{+,-\}$ and suppose $(\mathcal{F}_t)_{t\geq0}$ is a filtration. A pair of processes $(J,\xi)$ taking values in $E\times\R\cup\{-\infty\}$ with lifetime $\chi$ is a Lamperti-Kiu Markov Additive Process (MAP) with respect to $(\mathcal{F}_t)_{t\geq0}$ if, for any continuous bounded function $f:E\times\R\rightarrow\R^+$, $(z,y)\in E\times\R$ and $s,t\geq0$, we have
        \begin{equation}
            \label{eqn:mapDefn}
            \E_{z,y}[f(J_{t+s},\xi_{t+s}-\xi_t); t+s<\chi | \mathcal{F}_t ] = \E_{J_{t},0}[ f(J_s,\xi_s); s<\chi]\mathbbm{1}_{\{t<\chi\}}
        \end{equation}
    where $\Prob_{z,y}$ is the law of $(J,\xi)$ started at $(z,y)$ and $\E_{z,y}$ is the corresponding expectation. For all $t>\chi$, the process $\xi$ is in the cemetery state $-\infty$, whilst $J$ continues as a Markov chain. A detailed account of MAPs is given in \cite[Chapter XI]{asmussen2003applied} whilst a more general definition is given in \cite[Section 3, pp 10, Definition 1]{Alili16}. Note that $(J_t\exp(\xi_t),t\geq0)$ is a \cadlag multiplicative process taking values in $\R^*:=\R\notzero$ and so, following \cite{Rivero2011}, we refer to it as a Lamperti-Kiu process.

    There is a well known construction of a Lamperti-Kiu process given in \cite[pp 2502, Theorem 6(i)]{Rivero2011}.  Let $\xi^\pm$ be two \levy processes, $\zeta^\pm$ be two exponentially distributed random variables with rates $q^\pm$ and $U^\pm$ be two random variables taking values in $\R$. Then, consider sequences $(\xi^{\pm,k})_{k\in\N_0}:=\{0,1,2,\cdots\}, (\zeta^{\pm,k})_{k\in\N}$ and $(U^{\pm,k})_{k\in\N}$ of i.i.d. copies of $\xi^\pm$, $\zeta^\pm$ and $U^\pm$, respectively. Under $\Prob_{\sigma,x}$, that is assuming $(J_0,\xi_0)=(\sigma,x)$, for each $k\in\N$ let $\xi^k=\xi^{\gamma,k}$, $\zeta^k=\zeta^{\gamma,k}$ and $U^k=U^{\gamma,k}$ where $\gamma=\sigma(-1)^k$. Finally, let $\chi$ be an exponentially distributed random variable of rate $q\in[0,\infty)$ independent of the rest of the system, where $q=0$ is interpreted to mean $\chi=\infty$. Then, for $t\geq0$ and $x\in\R^*:=\R\notzero$, set
    \begin{align}
        \label{eqn:decompEps}
        Y_t &\coloneqq xJ_t\exp(\xi_t),
        \qquad t<\chi
    \end{align}
    with
    \begin{align*}
        \xi_t &:= \xi^{N_t}_{\pi_t} + \sum_{k=0}^{N_t-1}\left( \xi^k_{\zeta^k} + U^k \right)
        \qquad\text{and}\qquad
        J_t := \sigma\left(-1\right)^{N_t}
    \end{align*}
    where, for $n\in\N$,
    \begin{align*}
        T_0 &:=0, \qquad
	T_n        := \sum_{k=0}^{n-1}\zeta^k, \qquad
        N_t         := \max_{m\in\N_0}\left\{T_m \leq t\right\}
        \qquad\text{and}\qquad
        \pi_t    := t - T_{N_t}
    \end{align*}
    using the notation $\N_0\coloneqq\{0,1,2\cdots\}$. Then, $(Y_t,t\geq0)$ is a Lamperti-Kiu process and $(J,\xi)$ is the corresponding MAP. Conversely, any Lamperti-Kiu process has such a decomposition which we will refer to as the Lamperti-Kiu decomposition. We will refer to $\chi$ as the lifetime of the Lamperti-Kiu process.

    We study the standard and signed exponential functionals of $(Y_t,t\geq0)$ defined, respectively, by
    \begin{align}
        A_\infty := \int_0^\infty \exp(\xi_t) dt
        \qquad\text{and}\qquad
        B_\infty := \int_0^\infty Y_t dt.
    \end{align}

    Recall that a perpetuity is a security where a stream of cashflows is paid indefinitely, such as consols issued by the Bank of England. Under the MAP model, we suppose that the cashflows are paid continuously at a rate $c_t:=\exp(\xi_t+rt)$ at time $t\geq0$ where $r$ is the rate of interest. Each element of $E$ corresponds to a market state (for example its states may refer to a \textit{Bear} and \textit{Bull} market) where the state at time $t\geq0$ is given by $J_t$. The value of the perpetuity is given by $A_\infty$.

    It was shown in \cite{Rivero2011} that $A_\infty$ is the first hitting time of zero by an associated self-similar Markov process (ssMp). Indeed, consider the Lamperti-Kiu MAP $(J,\xi)$ from (\ref{eqn:mapDefn}) and for $\alpha\in(0,\infty)$ define the time transformation
    \begin{align*}
        \tau(t) \coloneqq \inf\left\{ u\geq0 \::\: \int_0^u\exp\left(\alpha\xi_s\right) ds \geq t\right\}.
    \end{align*}
    Then, for all $x\in\R^*$ the process $X_t^{(x)}:=J_{\tau(t|x|^{-\alpha})}x\exp(\xi_{\tau(t|x|^{-\alpha})})$ for $t<|x|^\alpha\int_0^\infty\exp\left(\alpha\xi_s\right)ds$ is a self-similar Markov process of index $\alpha$ taking values in $\R^*$ and started at $x$. That is, $X$ is a \cadlag Markov process such that, for all $c>0$ and $x\in\R^*$, it satisfies
    \begin{align*}
        \left( cX^{(x)}_{c^{-\alpha}t},t\geq0 \right) \eqL \left( X^{(cx)}_t,t\geq 0\right),
    \end{align*}
    where $\eqL$ means equality in law. Moreover, any self-similar Markov process taking values in $\R^*$ of index $\alpha$ can be constructed in this way. The first hitting time of zero by $X^{(x)}$ is $\int_0^\infty\exp(\alpha\xi_s)ds$ which equals  $A_\infty$ when $\alpha=1$. Many papers are devoted to the study of the Lamperti transform of self-similar Markov processes. For example see \cite{bertoin2005},  \cite{Rivero2011}, \cite{KIU1980183}, \cite{Lamperti1972} and \cite{Lamperti62}. Other applications of MAPs and their exponential functionals include multi-type self-similar fragmentation processes and trees, for example see \cite{stephenson2017expfunc}.

    The focus of this paper is on the finiteness and right tails of $A_\infty$ and $B_\infty$. The right tails of a distribution determine which positive moments exist and whether it is a member of the  classes of heavy-tailed, long-tailed or subexponential distributions. These classes of heavy tailed distribution are of particular interest to us because of their applications in finance, risk theory and insurance (for instance see\cite{ibragimov2015heavy}, \cite{rachev2003handbook}, \cite{resnick2007heavy}). Empirical data  often shows realised market returns to be heavy tailed (see \cite{cont2010empirical}, \cite{Jung2000empiricalheavy}). This motivates both considering Lamperti-Kiu processes with Lamperti-Kiu components which are heavy tailed and studying $A_\infty$ as an example of a heavy tailed distribution.

    %There is a large existing literature for heavy tailed distributions and the subclasses of long tailed, subexponential and strong subexponential, for example see \cite{foss2013introduction}. The implications of heavy tails in finance, insurance and risk theory is a much studied topic

    For \levy processes, which are MAPs where $E$ is a singleton set, the theory of the exponential functional is well developed. Several results on the moments and tails of the exponential functional, including random affine equations, are collected in the survey \cite{YorExponentialFunc}. Under \cramer
s condition, with \cramer number $\theta$, it is shown that the right tails are polynomial with order $-\theta$. More recently, \cite{patie2018} provided a complete description at the logarithmic level of the asymptotic of the right tail  and, under \cramers condition, the derivatives of the density.
    In the heavy tailed case, \cramers condition fails and different methods are needed. In this case, the right tails of $A_\infty$ have been studied in, for example, \cite{MaulikZwart2006156}, \cite{Palau2016asymp}, \cite{patie2018} and \cite{Rivero2009tailAsymp}. The case of a MAP is studied in \cite{Kuznetsov2014} where, under a \cramer type condition with a \cramer number $\theta\in(0,1)$, moments of order $s\in(0,1+\theta)$ are shown to exist and satisfy a recurrence relation. This leads to similar polynomial tails to the \levy case. The same recurrence relation is shown in \cite{stephenson2017expfunc} for the case where the additive component is not increasing. Other properties, including the finiteness and integer moments of the exponential functional of non-increasing MAPs are also given in \cite{stephenson2017expfunc}.

    We will use the law of large numbers and Erickson's law of large numbers for when the mean is undefined \cite{Erickson1973SLLN} to give a characterisation of the finiteness of $A_\infty$. Then, we show that for a Lamperti-Kiu MAP, both $A_\infty$  and $B_\infty$ satisfy a random affine equation. Under a \cramer type condition, we show that the conditions of the implicit renewal theorems of \cite{goldie1991} and \cite{kesten1973} hold, hence we are able to determine the right tails of $A_\infty$ and $B_\infty$. In the heavy tailed cases, when \cramers condition does not hold, a different approach to studying the tails of $A_\infty$ is required. We define a Lamperti-Kiu process to be of strong subexponential type when $Y_{T_2}$ is long tailed and one of the Lamperti-Kiu components, $\xi^{(\pm)}_1$ or $U^{(\pm)}$, is strong subexponential and has right tails which asymptotically dominate the right tails of the other Lamperti-Kiu components. By the careful consideration of an embedded Markov chain, we are able to overcome the lack of independent increments of $\xi$ and provide a generalisation of the subexponential results of \cite{MaulikZwart2006156} to Lamperti-Kiu processes of strong subexponential type. This provides an asymptotic expansion of the right tails of $A_\infty$ to show that $A_\infty$ is long tailed and $\log(A_\infty)$ is subexponential.

    This paper is organised as follows. In Section \ref{sec:finiteness}, we give necessary and sufficient conditions for the standard and signed exponential functionals $A_\infty$ and $B_\infty$, respectively, to be finite. In Section \ref{sec:moments}, we look at the random affine equation approach to studying the moments and tails of $A_\infty$ and $B_\infty$ under \cramers condition and the assumption that $Y_{T_2}$ doesn't have a lattice distribution.
    In Section \ref{sec:subexpTails}, we study the tails of $A_\infty$ when the Lamperti-Kiu process is of strong subexponential type.

\section{Finiteness of \texorpdfstring{$A_\infty$}{A infinity} and \texorpdfstring{$B_\infty$}{B infinity}}
    \label{sec:finiteness}

    Let us keep the mathematical setting of the introduction where
    $(Y_t,t\geq0)$ is the Lamperti-Kiu process defined in  (\ref{eqn:decompEps}) associated with the Lamperti-Kiu MAP $(J,\xi)$ given, for a fixed $t\geq0$, by $\xi_t\coloneqq\log|Y_t|$ and $J_t\coloneqq\sgn(Y_t)$.

    If possible, define $K\in\R\cup\{-\infty,+\infty\}$ by
    \begin{equation*}
        K := \frac{\E[\xi_{T_2}]}{\E[T_2]} =  \frac{q^-}{q^+ +  q^-}\left( \E[\xi_1^+] + q^+\E[U^+] \right) + \frac{q^+}{q^++q^-}\left( \E[\xi_1^-] + q^-\E[U^-] \right)
    \end{equation*}
    where we allow $K$ to take the values $+\infty$ and $-\infty$ but if both $\E[\max(\xi_{T_2},0)]=\infty$ and $\E[\max(-\xi_{T_2}^-,0)]=\infty$ we say that $K$ is undefined.

    A Lamperti-Kiu process $Y$ will be called \textit{degenerate} if $Y$ is such that $\limsup_{t\rightarrow\infty}\left|\xi_t\right|<\infty$. This can be shown to be equivalent to the case that either $Y$ has a finite lifetime or $\xi^\pm_t\equiv0$ for all $t\geq0$ and $U^+=-U^-$ is deterministic, hence,
    \begin{equation*}
        Y_t^{(x)} =   \begin{cases}
                        x                   &\text{ if } T_{2k\phantom{+1}}\leq t < T_{2k+1} \text{ for some } k\in\N_0\\
                        x\exp(U^{\sgn(x)})  &\text{ if } T_{2k+1}\leq t < T_{2k+2} \text{ for some } k\in\N_0
                    \end{cases}
    \end{equation*}
    for all $t\geq0$ and $x\in\R^*$.

    When $K$ is defined and $Y$ has an infinite lifetime from \cite[pp 214, Proposition 2.10]{asmussen2003applied} and \cite[pp 313, Corollary 2.8]{asmussen2003applied} it is known that almost surely
    $\lim_{t\rightarrow\infty}t^{-1}\xi_t=K$. Also,  if $K=0$ and $Y$ is non-degenerate then $\lim_{t\rightarrow\infty}t^{-1}\xi_t=0$, $\liminf_{t\rightarrow\infty}\xi_t = -\infty$ and $\limsup_{t\rightarrow\infty}\xi_t=\infty$, almost surely.

    In the case where $K$ is undefined, we will use Erickson's theorem \cite[pp 372, Theorem 2]{Erickson1973SLLN}, which provides a strong law of large numbers for a random walk with an undefined mean. The following lemma provides an analogous result to Erickson's theorem for MAPs.

    First, we define:
    \begin{align*}
        m_-(x) &\coloneqq \int_{-x}^0\Prob(\xi_{T_2}\leq y) dy, \qquad
        m_+(x) \coloneqq \int_0^x\Prob(\xi_{T_2}>y) dy
    \end{align*}
    and
    \begin{align*}
        I_+ &\coloneqq \int_0^\infty \frac{x}{m_-(x)}\Prob(\xi_{T_2}\in dx), \qquad
        I_- \coloneqq \int_{-\infty}^0 \frac{|x|}{m_+(|x|)}\Prob(\xi_{T_2}\in dx).
    \end{align*}
    Then, the long term behaviour of $(\xi_t,t\geq0)$ is described by the following lemma.
    \begin{lem}
        \label{lem:ericksonMAP}
        Suppose that $K$ is undefined. Then, at least one of $I_+$ and $I_-$ equals infinity and  almost surely, we have:
        \begin{enumerate}[(i)]
            \item $\limsup_{t\rightarrow\infty}t^{-1}\xi_t=\infty$ if and only if $I_+=\infty$;
            \item $\lim_{t\rightarrow\infty}t^{-1}\xi_t=\infty$  if and only if $I_+=\infty$ and $I_-<\infty$;
            \item $\liminf_{t\rightarrow\infty}t^{-1}\xi_t=-\infty$  if and only if $I_-=\infty$;
            \item $\lim_{t\rightarrow\infty}t^{-1}\xi_t=-\infty$  if and only if $I_+<\infty$ and $I_-=\infty.$
        \end{enumerate}
    \end{lem}
    \begin{proof}
        Consider the sequence $\left\{\xi_{T_{2n}}\right\}_{n\in\N}$ as the random walk
        \begin{align*}
            \xi_{T_{2n}} = \sum_{k=1}^{n} \left( \xi_{T_{2k}} - \xi_{T_{2k-2}} \right)
        \end{align*}
        and notice that
        $\xi_{T_{2n}} - \xi_{T_{2(n-1)}} \eqL \xi_{T_2}$ has an undefined mean, for each $n\in\N$.
        Then,  Erickson's theorem for random walks with undefined mean \cite[pp 372, Theorem 2]{Erickson1973SLLN} and the remark which follows it, state that either $I_+=\infty$ or $I_-=\infty$ or both hold, proving the first statement of the lemma. Furthermore, the following statements hold:
        \begin{enumerate}[(1)]
            \item $\limsup_{n\rightarrow\infty}n^{-1}\xi_{T_{2n}}=\infty$ a.s. if and only if $I_+=\infty$;
            \item $\lim_{n\rightarrow\infty}n^{-1}\xi_{T_{2n}}=\infty$ a.s. if and only if $I_+=\infty$ and $I_-<\infty$;
            \item $\liminf_{n\rightarrow\infty}n^{-1}\xi_{T_{2n}}=-\infty$ a.s. if and only if $I_-=\infty$;
            \item $\lim_{n\rightarrow\infty}n^{-1}\xi_{T_{2n}}=-\infty$ a.s. if and only if $I_+<\infty$ and $I_-=\infty$;
        \end{enumerate}
        and similar statements hold for $\{T_{2n+1}\}_{n\in\N}$.

        Since $\E[T_2]<\infty$ it is immediate that if $\limsup_{n\rightarrow\infty}n^{-1}\xi_{T_{2n}}=\infty$ a.s. then $\limsup_{t\rightarrow\infty}t^{-1}\xi_t=\infty$ a.s.  also and if $\liminf_{n\rightarrow\infty}n^{-1}\xi_{T_{2n}}=-\infty$ a.s.  then $\liminf_{t\rightarrow\infty}t^{-1}\xi_t=-\infty$ a.s.  hence the ``if'' direction of statements (i) and (iii) holds. To prove the ``only if'' direction of (i) and (iii) we must first prove (ii) and (iv).

        Consider (iv) and notice that the ``only if'' direction is an immediate consequence of statement (4) above. Now, suppose $I_+<\infty$ and $I_-=\infty$. Then, $\limsup_{n\rightarrow\infty}n^{-1}\xi_{T_{2n}}=-\infty$ a.s. and $\limsup_{n\rightarrow\infty}n^{-1}\xi_{T_{2n+1}}=-\infty$ a.s.  hence $\limsup_{n\rightarrow\infty}n^{-1}\xi_{T_n}=-\infty$ a.s.  also. Suppose for a contradiction there exists an $R>0$ such that $\limsup_{t\rightarrow\infty}t^{-1}\xi_t>-R>-\infty$ a.s. Since $\lim_{n\rightarrow\infty}n^{-1}T_n = \frac{1}{2}\E[T_2]$ a.s. and $\limsup_{n\rightarrow\infty}n^{-1}\xi_{T_n}=-\infty$ a.s. there exists some $N\in\N$ such that for all $n>N$ we have $T_n>1$ and
        \begin{align*}
            \frac{\xi_{T_n}}{T_n} < -2R \quad\text{a.s.}
        \end{align*}
        Define a sequence $(\tau_n)_{n\in\N}$ of times and $(x_n)_{n\in\N}$ of values such that for each $n\in\N$ we have
        \begin{align*}
            \tau_n = \sup\left\{t\in[T_n,T_{n+1}) \::\: \xi_t = \sup_{s\in[T_n,T_{n+1})}\xi_s \right\}
            \qquad\text{and}\qquad
            x_n = \sup\left\{\xi_t\::\: t\in[T_n,T_{n+1})\right\}.
        \end{align*}
        Since $\limsup_{t\rightarrow\infty}t^{-1}\xi_t>-R$ a.s., there is an increasing sequence of times $(s_n)_{n\in\N}$ such that $s_n^{-1}\xi_{s_n}>-R$ a.s. for each $n\in\N$ and $\lim_{n\rightarrow\infty}s_n=\infty$. Then, we may take a subsequence $(s_n')_{n\in\N}$ such that there is at most one element of the sequence $(s_n')_{n\in\N}$ in each interval $[T_m,T_{m+1}]$ and $J_{s_n'}$ is constant.

        Let $(\tau_{k_n})_{n\in\N}$ be a subsequence of $(\tau_{n\in\N})$ such that $k_n>N$ and $s_n'\in[T_{k_n},T_{{k_n}+1}]$ for each $n\in\N$. Then, for each $n\in\N$ we have
        \begin{align*}
            x_{k_n}\geq\xi_{s_n'}>-Rs_n'\geq-R T_{k_n+1}
        \end{align*}
        whilst $\xi_{T_{k_n+1}}<-2RT_{k_n+1}$ and $T_{k_n+1}>1$ therefore
        \begin{align*}
            x_{k_n}-\xi_{T_{k_n}+1} > RT_{k_n+1}>R,
        \end{align*}
        almost surely. However, for each $n\in\N$, by the Lamperti-Kiu decomposition, $\{\xi_{T_{k_n}+t}-\xi_{T_{k_n}} \::\:t<T_{{k_n}+1}-T_{k_n}\}$
        is a \levy process
        and so, by splitting at the last time of the maximum
        \cite[pp 160, Chapter VI, Theorem 5]{bertoin1998levy},
        %\cite[Theorem 1, pp2]{millar1978},
        $x_{k_n}-\xi_{T_{k_n+1}}$ is independent of $x_{k_n}$ and has the same distribution as $x_0-\xi_{T_1}$ hence its distribution doesn't depend on $R$. This contradicts the fact that it only has support on $(R,\infty)$. Hence $I_+<\infty$ and $K$ undefined imply that $\lim_{t\rightarrow\infty}t^{-1}\xi_t=-\infty$ a.s. and so the ``if'' direction of (iv) holds. By applying similar arguments to $-\xi_t$, statement (iii) of the lemma also holds.

        To prove the ``only if'' direction of (i) suppose $I_+<\infty$. Since $K$ is undefined, by \cite[pp 372, Theorem 2]{Erickson1973SLLN} and the remark which follows it, we have that $I_-=\infty$. Then, by statement (iv), we have $\lim_{t\rightarrow\infty}t^{-1}\xi_t=-\infty$ a.s. Hence, $\limsup_{t\rightarrow\infty}t^{-1}\xi_t=\infty$ a.s. only if $I_+=\infty$. The argument for (iii) is analogous.
    \end{proof}

    The following theorem shows that in the infinite lifetime case the convergence and finiteness of $A_\infty$ and $B_\infty$ is fully characterised by $K$ when this is defined and by the finiteness of $I_+$ otherwise.
    \begin{thm}%[Finiteness of Exponential Functional]
        \label{thm:convAbs}
        Suppose $Y$ has an infinite lifetime. Then, $A_\infty$ converges  a.s. if and only if $B_\infty$ converges a.s. Moreover, $A_\infty$ and $B_\infty$ converge almost surely  if and only if either $K$ is defined and $K<0$ or $K$ is undefined and $I_+<\infty.$
    \end{thm}

    Before proving Theorem \ref{thm:convAbs}, we prove the following preliminary Lemma.

    \begin{lem}
        \label{lem:limsupxiT2n}
        If $\limsup_{n\rightarrow\infty}\xi_{T_{2n}}=\infty$ a.s. then both $A_\infty$ and $B_\infty$ diverge a.s.
    \end{lem}

    \begin{proof}
        If $\limsup_{n\rightarrow\infty}\xi_{T_{2n}} = \infty$ a.s. then there exists a strictly increasing sequence $\left\{\tau_n\right\}_{n\in\N}\subset\N$ such that $\exp\left(\xi_{T_{2\tau_n}}\right) \geq 1$ a.s. for all $n\in\N$. First, by considering $A_\infty$ and using the Markov property for the second inequality, we have
        \begin{align*}
            A_\infty
            &\geq \sum_{n=0}^\infty \exp\left( \xi_{T_{2\tau_n}} \right)\int_{T_{2\tau_n}}^{T_{2\tau_n+2}}\exp\left( \xi_t - \xi_{T_{2\tau_n}}\right) dt
            \geq \sum_{n=0}^\infty \int_{0}^{T_2^{(n)}}\exp\left(\xi_t^{(n)}\right) dt.
        \end{align*}
        Since the terms of this series form an i.i.d sequence with strictly positive values, the series must diverge.
	%$\left\{\left(\xi_t^{(n)},t\geq0\right)\right\}_{n\in\N}$ and $\left\{ T_{2}^{(n)} \right\}_{n\in\N}$ are sequences of i.i.d. copies of $(\xi_t,t\geq0)$ and $T_2$, respectively. However, the last term of the above expression is a sum of strictly positive i.i.d. terms and so the series must diverge.

        Similarly, $B_\infty$ converges a.s. only if the sum
        $\sum_{n=0}^\infty \int_{T_{2n}}^{T_{2n+2}}J_t \exp(\xi_t) dt$
        converges a.s. which implies the a.s. convergence to zero of the subsequence
        \begin{align*}
            \left|\int_{T_{2\tau_n}}^{T_{2\tau_n+2}}J_t \exp(\xi_t) dt\right|
            = \exp\left( \xi_{T_{2\tau_n}} \right)\left|\int_{T_{2\tau_n}}^{T_{2\tau_n+2}} J_t\exp\left( \xi_t - \xi_{T_{2\tau_n}} \right) dt\right|
            \geq \left|\int_{T_{2\tau_n}}^{T_{2\tau_n+2}} J_t\exp\left( \xi_t - \xi_{T_{2\tau_n}} \right) dt\right|.
        \end{align*}
        Then, by using the Markov property, $B_\infty$ converges a.s.  only if
        \begin{align*}
            \left| \int_0^{T_2^{(n)}} J_t^{(n)} \exp\left(\xi_{t}^{(n)}\right) dt\right| \rightarrow 0
            \qquad\text{a.s.}
        \end{align*}
        as $n\rightarrow\infty$. This convergence is impossible since we are dealing with an i.i.d. sequence which doesn't converge to zero in distribution.
    \end{proof}

    \begin{proof}[Proof of Theorem \ref{thm:convAbs}]
        We will consider the cases that $K$ is defined and undefined separately.

        \textbf{1.}
        Suppose that $K$ is defined and first consider $K<0$. Then, recall that $\lim_{t\rightarrow\infty}t^{-1}\xi_t=K$ a.s. as $t\rightarrow\infty$. If $-\infty<K<0$ let $k=\frac{1}{2}K$ and if $K=-\infty$ let $k=-1$. Then a.s. there exists a finite $T\geq0$ such that $\xi_t< kt <0$ for all $t>T$ thus $A_\infty\leq \int_0^T\exp(\xi_t)dt + k^{-1}e^{kT}<\infty$ a.s. and by absolute convergence it is then immediate that $B_\infty$ also converges a.s..

        Next, we consider the case that either $K>0$ or both $K=0$ and $Y$ is non-degenerate. Then by \cite[pp 167, Chapter 9, Proposition 9.14]{kallenberg2002foundations}  $\limsup_{n\rightarrow\infty}\xi_{T_{2n}} = \infty$ a.s. and so the result follows from Lemma \ref{lem:limsupxiT2n}. If $K=0$ and $Y$ is degenerate, then, since $Y$ has an infinite lifetime, we have
        \begin{equation*}
            \xi_t = \begin{cases}
                    0 &\text{ if } T_{2k} \leq t < T_{2k+1} \text{ for }k\in\N_0;\\
                    U^{J_0} &\text{ if } T_{2k+1} \leq t < T_{2k+2} \text{ for }k\in\N_0.
                  \end{cases}
        \end{equation*}
        Thus, for all $t\geq0$, we have
        \begin{align*}
            e^{\xi_t} \geq \min\left(1, \exp\left( U^{J_0} \right) \right) =: V >0
        \end{align*}
        and so $A_\infty=\infty$ a.s.
                        \if{}
                        \begin{align*}
                            A_\infty = \int_0^\infty e^{\xi_t} dt \geq \int_0^\infty V dt = \infty.
                        \end{align*}
                        \fi
        Also, $B_\infty$ can be written as the sum
        \begin{align*}
            B_\infty
            = \sum_{n=0}^\infty \left(\int_{T_{2n}}^{T_{2n+1}} J_te^{\xi_t} dt + \int_{T_{2n+1}}^{T_{2n+2}} J_te^{\xi_t} dt\right)
            = J_0\sum_{n=0}^\infty \left(\zeta^{2n} - \exp(U^{J_0})\zeta^{2n+1}  \right)
        \end{align*}
        and since $\zeta^{2n} - \exp(U^{J_0})\zeta^{2n+1} $ does not converge to zero in distribution $B_\infty$ must a.s. diverge.

        \textbf{2.} Suppose that $K$ is undefined. From \cite[pp 372, Theorem 2]{Erickson1973SLLN}, we know that if $I_+=\infty$ then $\limsup_{n\rightarrow\infty}\xi_{T_{2n}}=\infty$ a.s. Hence, by using Lemma \ref{lem:limsupxiT2n}, both $A_\infty$ and $B_\infty$ diverge a.s.

        If $I_+<\infty$ then since $K$ is undefined, as a consequence of \cite[pp 372, Theorem 2]{Erickson1973SLLN}, $I_-=\infty$ and so by Lemma \ref{lem:ericksonMAP} we have $\limsup_{t\rightarrow\infty}t^{-1}\xi_t=-\infty$ a.s. Then by the argument of case 1. above both $A_\infty$ and $B_\infty$ converge a.s. as desired.

    \end{proof}

\section{Moments and tail asymptotics of \texorpdfstring{$A_\infty$}{A infinity} and \texorpdfstring{$B_\infty$}{B infinity}}
\label{sec:moments}
    Throughout this section, we assume the Lamperti-Kiu process has an infinite lifetime. For $z\in\C$ suppose the characteristic exponents $\psi_\pm(z) := \log\left(\E[\exp(z\xi_1^\pm)]\right)$ and Laplace transforms $G^\pm(z):=\E[\exp(zU^\pm)]$ exist and are finite. Then the matrix exponent of $Y$ is defined to be
    \begin{equation*}
        F(z) := \left(\begin{array}{cc}
                    \psi_+(z)   & 0\\
                    0           & \psi_-(z)
                \end{array}\right)
                +
                \left( \begin{array}{cc}
                    -q_+        & q_+G^+(z) \\
                    q_-G^-(z)   & -q_-
                \end{array}\right)
    \end{equation*}
    and from \cite[pp 311, XI.2b]{asmussen2003applied} it is known that, for $l,j\in\{+,-\}$ and $z\in\C$,
    \begin{equation}
        \E\left[ e^{z\xi_t}; \: J_t = j \:\middle|\: J_0=l \right] = \left(e^{tF(z)}\right)_{l,j}.
    \end{equation}
    Then, let $\lambda(z)$ denote the eigen value of $F(z)$ with largest real part. Using Perron-Frobenius theory, it is shown in  \cite[pp 8, Proposition 3.2]{Kuznetsov2014} that such an eigen value is guaranteed to be simple, real and continuous as a function of $z$. From \cite[Section 3, pp 10, Proposition 3.4]{Kuznetsov2014} we also know that $\lambda(z)$ is convex when considered as a map $\lambda:\R\rightarrow\R$. Provided $F$ exists in a neighbourhood of zero, from \cite[pp 313, Corollary 2.9]{asmussen2003applied}, it is known that $K=\lambda'(0)\in[-\infty,\infty]$, where the derivative is considered on this restriction.

    It can be shown that the integrals $A_\infty$ and $B_\infty$ are the solutions of random affine equations. We will assume $q_+,q_->0$ so that $T_2<\infty$ and note that $T_2$ is independent of $Y_0$. Then, for any $T\in[0,\infty]$, define the random variables
    \begin{align*}
        A_T \coloneqq \int_0^T e^{\xi_t}\, dt
        \qquad\text{and}\qquad
        B_T \coloneqq \int_0^T J_te^{\xi_t}\, dt,
    \end{align*}
    and notice that $A_\infty=\lim_{T\rightarrow\infty}A_T$ and $B_\infty=\lim_{T\rightarrow\infty}B_T$.
    Similarly, due to the result in \cite[Section 4.3]{bertoin2005},
    by the Markov additive property
    \begin{align*}
        B_\infty
        =B_{T_2} + Y_{T_2}\hat{B}_\infty,
    \end{align*}
    where $\hat{B}_\infty$ is an independent copy of $B_\infty$%., which is independent of $B_{T_2}$ and $Y_{T_2}$.

    Notice that $Y_{T_2}$ has the same sign as $Y_0$ and that $|Y_{T_2}|$ is independent of $J_0$ because of its symmetry in the components of the decomposition of $Y$. Similarly,
    \begin{equation*}
        A_\infty = A_{T_2} + Y_{T_2}\hat{A}_\infty,
    \end{equation*}
    where $\hat{A}_\infty$ is an independent but identically distributed copy of $A_\infty$ and is independent of $Y_{T_2}$ and $A_{T_2}$.

    The following results are generalisations of \cite[pp 201, Corollary 5]{bertoin2005} to Lamperti-Kiu processes using the implicit renewal theorems given in \cite[Theorem 4.1, pp 135]{goldie1991} and \cite[Theorem 5, pp 246]{kesten1973}. In the case of a \levy process, it has been shown in \cite{patie2018} that the coefficient $c_A$ of the Proposition below can be found explicitly by evaluating a Bernstein-gamma function.

    \begin{prop}
        \label{prop:kestenTails}
        Suppose $Y$ is an infinite lifetime Lamperti-Kiu process with $K<0$ and there is a $\kappa>0$ such that $F(\kappa)$ exists, $\lambda(\kappa)=0$ and
        \begin{align}
            \label{eqn:tails:kestenCondition}
            \E\left[|Y_{T_2}|^\kappa\log^+|Y_{T_2}|\right]<\infty.
        \end{align}
        If $Y_{T_2}$ does not have a lattice distribution, then there exist constants $c_A,c_B^+,c_B^-\in\R$ such that
        \begin{align*}
            c_A:=\lim_{t\rightarrow\infty}t^\kappa\Prob(A_\infty>t)
            ,\quad c^+_B := \lim_{t\rightarrow\infty}t^\kappa\Prob(B_\infty>t)
            \quad \text{and} \quad
            c^-_B :=\lim_{t\rightarrow\infty}t^\kappa\Prob(B_\infty<-t),
        \end{align*}
        hence $A_\infty,B_\infty$ have moments of order $s\in\C^+:=\{x\in\C : \Re(x)\geq0 \}$ if and only if $0\leq \Re(s)<\kappa$.
    \end{prop}
    \begin{proof}
        Since the proposition assumes (\ref{eqn:tails:kestenCondition}), the result is an immediate consequence of \cite[pp 129, Section 2, Theorem 2.3]{goldie1991} and \cite[pp 246, Section 4, Theorem 5]{kesten1973} provided:
        \begin{align}
            \E[\log|Y_{T_2}|] &<0,                \label{eqn:golide1}       \\
            \E[|Y_{T_2}|^\kappa]&=1,              \label{eqn:golide2}     \\
            %\E[|Y_{T_2}|^\kappa\log^+|Y_{T_2}|] &< \infty,  \label{eqn:golide3} \\
            0<\E[|A_{T_2}|^\kappa]&<\infty,           \label{eqn:golide4a} \\
            0<\E[|B_{T_2}|^\kappa]&<\infty.           \label{eqn:golide4b}
        \end{align}
        We now prove that under the conditions of the proposition each of these equations holds.

        To show equation (\ref{eqn:golide1}), we expand $\log|Y_{T_2}|$ to get
        \begin{align*}
            \E\left[\log|Y_{T_2}|\right]
                = \frac{1}{q^+}\E[\xi_1^+] + \frac{1}{q^-}\E[\xi_1^-] + \E[U^+] + \E[U^-]
                =\left(\frac{1}{q_+} + \frac{1}{q_-} \right)K
        \end{align*}
        and since $q_+^{-1} + q_-^{-1}>0$ equation (\ref{eqn:golide1}) follows from the assumption $K<0$.

        Since $\det(F(z)) = (\psi^+(z) - q^+)(\psi^-(z) - q^-) - q^+q^-G^+(z)G^-(z)$ and by the assumption that $0$ is an eigenvalue of $F(\kappa)$, we have that
        \begin{align}
            \label{eqn:0eigenValue}
            1 = \left(\frac{q^+G^+(\kappa)}{\psi^+(\kappa) - q^+}\right)\left(\frac{q^-G^-(\kappa)}{\psi^-(\kappa)-q^-}\right).
        \end{align}

        Let $\mu(z)$ be the other eigenvalue of $F(z)$. Then, for all real $z\in(0,\kappa)$, by assumption, $\mu(z)<\lambda(z)<0$ and so $0 < \mu(z)\lambda(z) = \det(F(z))$. Rearranging this gives $(\psi^+(z) - q^+)(\psi^-(z) - q^-) > 0$, hence, $\psi^\pm(z) - q^\pm$ has no roots in $(0,\kappa)$. Since $\psi^\pm(0)-q^\pm<0$, by continuity, $\psi^\pm(z)-q^\pm<0$ for all $z\in(0,\kappa)$. By independence and using (\ref{eqn:0eigenValue}), we get
        \begin{align*}
            \E\left[\left|Y_{T_2}\right|^\kappa\right]
            &= \E[\exp(\psi^+(\kappa){\zeta^+})]G^+(\kappa)\E[\exp(\psi^-(\kappa){\zeta^-})]G^-(\kappa)=1,
        \end{align*}
        hence equation (\ref{eqn:golide2}) holds.

       Using independence and the inequality $(a+b)^x\leq 2^x(a^x+b^x)$ for $a,b,x>0$, we get
        \begin{align*}
            \E\left[ |A_{T_2}|^\kappa \right]
            \leq 2^\kappa\left( \E\left[  \left(\int_0^{\zeta^1} \exp(\xi_s^1)ds \right)^\kappa\right] + \E\left[\exp\left(\kappa\left(\xi^1_{\zeta^1} + U^1 \right)\right)\right] \E\left[ \left(\int_0^{\zeta^2}\exp(\xi_s^2)ds\right)^\kappa \right] \right).
        \end{align*}
        From \cite{bertoin2005} it is know that
        $
            \E\left[\left( \int_0^{\zeta_\pm}\exp(\xi^\pm_s) ds \right)^x\right] < \infty
        $
        for all $x\in(0,\infty)$ such that $\psi^\pm(x)-q^\pm<0$. This follows from the fact $\int_0^{\zeta^\pm}\exp(\xi^\pm_s) ds$ is the exponential functional of the \levy process $\xi^{(\pm)}$ sent to the cemetery state $-\infty$ at an independent, exponentially distribution time of rate $q^\pm$.
        Then, since we have already seen that $\psi^\pm(\kappa)-q^\pm<0$, it follows that
        $
            \E\left[\left( \int_0^{\zeta^\pm}\exp(\xi^\pm_s) ds \right)^\kappa\right] < \infty.
        $
        By the assumption that $F(\kappa)$ exists, we have $\E\left[ \exp(\kappa U^\pm) \right]<\infty$, whilst $\E\left[ \exp(\kappa\xi_{\zeta_\pm}^\pm) \right]<\infty$ follows easily from $ \E[\exp(\kappa\xi_t^\pm)]=\exp(t\psi_\pm(\kappa))$ for $t\geq 0$. Hence,
        $\E\left[|B_{T_2}|^\kappa\right]\leq\E\left[ |A_{T_2}|^\kappa \right]<\infty$
        and so equations (\ref{eqn:golide4a}) and (\ref{eqn:golide4b}) hold.
\end{proof}

\begin{rmk}
        In the case that $G^{\pm}$ are continuous, (\ref{eqn:tails:kestenCondition}) is automatic.
       Indeed, by continuity, we can pick $\epsilon>0$ such that $\psi^\pm(\kappa+\epsilon)-q^\pm<0$ and $G^\pm(\kappa+\epsilon)<\infty$. Then, from the proof of equation (\ref{eqn:golide2}) we obtain
        \begin{align*}
            \E[|Y_{T_2}|^{\kappa+\epsilon}] = \left( \frac{q^+G^+(\kappa+\epsilon)}{q^+-\psi^+(\kappa+\epsilon)} \right)\left( \frac{q^-G^-(\kappa+\epsilon)}{q^--\psi^-(\kappa+\epsilon)} \right)<\infty.
        \end{align*}
        Since $\log(x)^+<x^\epsilon$ for all $x\geq R$, for some sufficiently large $R>0$, we have
        \begin{equation*}
            \E[|Y_{T_2}|^\kappa\log Y_{T_2}^+]
            = \E[|Y_{T_2}|^\kappa\log Y_{T_2}^+\:;\: |Y_{T_2}|\leq R] + \E[|Y_{T_2}|^\kappa\log Y_{T_2}^+\:;\: |Y_{T_2}|> R]
            \leq R^{\kappa+\epsilon} + \E[|Y_{T_2}|^{\kappa+\epsilon}]<\infty ,
        \end{equation*}
        hence equation (\ref{eqn:tails:kestenCondition}) holds.
\end{rmk}

\section{Subexponential tails of \texorpdfstring{$A_\infty$}{A infinity}}
    \label{sec:subexpTails}
    When the conditions of Proposition \ref{prop:kestenTails} do not hold, a different approach to the investigation of the tails and moments of $A_\infty$ is required. In Proposition \ref{prop:kestenTails}, it is assumed that $F(\kappa)$ exists, which requires that positive exponential moments of $\xi$ must exist. This is a condition that doesn't necessarily hold in general, and in particular, when $\xi_{T_2}$ is heavy tailed.

    In this section we first present a framework for bounding $\log(A_\infty)$ by considering a piecewise linear bound for $\{\xi_t:t\geq0\}$ with a.s. finitely many discontinuities. We then define Lamperti-Kiu processes of strong subexponential type and use this framework in conjunction with the subexponential properties to obtain the right tails of $A_\infty$. Interestingly, the resulting tails are of a very different nature to those considered under \cramers condition.

    We will need the following lemma, which bounds the probability distribution of the supremum of a \levy process over an exponentially distributed interval of time in terms of the distribution of the \levy process at the end of the interval. It is a variation of \cite[Lemma 1]{WILLEKENS1987173}, where the time interval considered was fixed.
    \begin{lem}
        \label{lem:WillekensExtension2}
        Let $X$ be a \levy process, $\tau$ be an independent exponentially distributed random variable and suppose $0<u_0<u$. Then,
        \begin{equation}
            \label{eqn:willkens}
            \Prob\left( \sup_{0\leq s<\tau} X_s > u  \right) \leq\frac{\Prob\left(X_\tau \geq u-u_0 \right)}{\Prob\left( X_\tau \geq -u_0 \right)}.
        \end{equation}
    \end{lem}
    \begin{proof}
        Let $S_u :=\inf\{t\geq0\::\:X_t>u\}$. Then, since $X_{S_u}\geq u$, by independent increments of $X$ and the memoryless property of $\tau$, we have that
        \begin{align*}
            \Prob(S_u<\tau ; X_{\tau} < u-u_0 )
            &\leq \Prob( S_u<\tau; X_{\tau} - X_{S_u} < -u_0 )
            =\Prob( S_u<\tau; \tilde{X}_{\tilde{\tau}} < -u_0 )
            = \Prob(S_u<\tau)\Prob(X_{\tau}<-u_0),
        \end{align*}
        %Independent increments of $X$ and the memoryless property of $\tau$ gives
        %\begin{align*}
        %    \Prob(S_u<\tau ; X_{\tau} < u-u_0 )
        %    &\leq ,
        %\end{align*}
        where $\tilde{X}$ and $\tilde{\tau}$ are independent and identically distributed copies of $X$ and $\tau$, respectively. Then, (\ref{eqn:willkens}) can be obtained by rearranging the inequality
        \begin{align*}
            \Prob(S_u<\tau)
            &\leq \Prob\left(X_{\tau}\geq u-u_0\right) + \Prob\left( S_u<\tau ;\: X_{\tau}<u-u_0 \right)
            \leq \Prob\left(X_{\tau}\geq u-u_0\right) + \Prob(S_u<\tau)\Prob(X_{\tau}<-u_0).
        \end{align*}
        %By rearranging, we obtain
        %\begin{align*}
        %    \Prob\left( \sup_{0\leq s<\tau}X_s > u \right) = %\Prob\left(S_u<\tau\right) \leq \frac{\Prob\left(X_\tau \geq u-u_0 \right)}{ 1 %- \Prob\left( X_\tau < -u_0 \right)}.
        %\end{align*}
    \end{proof}

    We now consider long-tailed distributions. Let $Q:\R^+\rightarrow[0,1]$ be a probability distribution and $\bar{Q}(x)\coloneqq 1-Q(x)$ for all $x\in\R^+$. Then, $Q$ is a long-tailed distribution if $Q(x+y)/Q(x)\rightarrow 1$ as $x\rightarrow\infty$, for any $y\in\R^+$.
    For any two functions $f,g:\R^+\rightarrow\R^+$, we will write
        $f\sim g$ if $\lim_{x\rightarrow\infty}f(x)/g(x) = 1$ and say that $f$ and $g$ are asymptotically equivalent.

	For a random variable $X$, we define the functions
    \begin{align}
        \label{eqn:defnIntegratedTail}
        G_X(x) := \int_x^\infty \Prob(X>u) du
        \qquad\text{and}\qquad
        H_X(x) := \min\left( 1, G_X(x) \right)
    \end{align}
    and refer to $H_X$ as the integrated tail of $X$.

    %To adapt \cite[pp 11, Lemma 4.2]{MaulikZwart2006156} to the Lamperti-Kiu case, we need to show the integrated tail is asymptotically equivalent to an infinite series. In the next lemma, we show this for a general long-tailed random variable.
        In the next lemma we show that the integrated tail of a long-tailed random variable is asymptotically equivalent to an infinite series. This will be used in Lemma \ref{lem:limZalphaBeta} to show the asymptotic equivalence of two distributions. %The method of proof is adapted from \cite[pp 11, Lemma 4.2]{MaulikZwart2006156}.

    \begin{lem}
        \label{lem:tailsumdecomp}
        Suppose $K+\epsilon<0$ and that $X$ is a long tailed random variable, which is independent of $(T_{2n})_{n\in\N}$. Then, the integrated tail of $X$ has the asymptotics
        \begin{equation}
            \label{eqn:tailSumDecomp}
            \int_x^\infty \Prob(X>u)\,du
            \sim \E[T_2]|K+\epsilon| \sum_{n=0}^\infty \Prob( X > x- (K+\epsilon)T_{2n}),
        \end{equation}
        where $G_X$ is the function defined in (\ref{eqn:defnIntegratedTail}).
    \end{lem}

    \begin{proof}
        By splitting the interval $(0,\infty)$ into a disjoint union, we can write
        \begin{align*}
            G_X(x)
            &= \sum_{n=0}^\infty  \int_{-T_{2n}(K+\epsilon)}^{-T_{2(n+1)}(K+\epsilon)} \Prob(X>u_1+x) du_1.
        \end{align*}
        Then, by using the change of variables $u=-(K+\epsilon)^{-1}u_1$, we get
        \begin{align*}
            G_X(x) &= \sum_{n=0}^\infty  \int_{T_{2n}}^{T_{2(n+1)}} \Prob\left( X > -(K+\epsilon)u + x\right) \: |K+\epsilon|  du.
        \end{align*}
        By independence of $\{T_{2n}\}_{n\in\N}$ and $X$, we can shift the domain of integration to obtain
        \begin{align*}
            G_X(x) &= \sum_{n=0}^\infty  \int_0^{T_{2(n+1)}-T_{2n}} \Prob\left(\vphantom{\Big(} X > x - (K+\epsilon)( u + T_{2n} ) \:\middle|\: T_{2n} \right)  \:|K+\epsilon|\: du.
        \end{align*}
        Taking expectations and noting that the left hand side is not random, that $T_{2(n+1)}-T_{2n}\eqL T_2$ and that $T_{2(n+1)}-T_{2n}$ is independent of $T_{2n}$ gives
        \begin{align*}
            G_X(x)
            &=|K+\epsilon|\sum_{n=0}^\infty \E\left[ \int_0^{\tilde{T}_2} \Prob\left( \vphantom{\Big(} X > x-(K+\epsilon)(u + T_{2n}) \:\middle|\: T_{2n} \right)  \: du \right],
        \end{align*}
        where $\tilde{T}_{2}$ is an independent and identically distributed copy of $T_2$. This can be written in the integral form
        \begin{align}
            \label{eqn:tailDecompIntegrals}
            G_X(x)= |K+\epsilon|\sum_{n=0}^\infty \int_0^\infty \Prob(\tilde{T}_2 \in ds )
            \int_0^s \int_0^\infty \Prob(T_{2n}\in dv) \Prob\left(X> x - (K+\epsilon)(u + v)\right) du.
        \end{align}

        Let $\delta>0$ then, since $X$ is long-tailed, for all $s>0$ there exists $R(s)>0$ such that, whenever $z>R(s)$ and $y\in[0,-s(K+\epsilon)]$,
        \begin{equation*}
            (1-\delta) \leq \frac{\Prob( X > z+ y )}{\Prob(X>z)} \leq (1+\delta)
        \end{equation*}
        and since $-v(K+\epsilon)\geq 0$ for $v\geq0$ we have for all $x>R(s)$ and $u\in[0,s]$
        \begin{align*}
            (1-\delta)   \leq \frac{\Prob( X > x-(K+\epsilon)(v+u) )}{\Prob( X> x - (K+\epsilon)v )} \leq (1+\delta) .
        \end{align*}
        To show the lower bound we use this inequality within the last two integrals of (\ref{eqn:tailDecompIntegrals}) to obtain, for $x>R(s)$,
        \begin{multline*}
            \int_0^s \int_0^\infty \Prob(T_{2n}\in dv) \Prob\left(X> x - (K+\epsilon)(u + v)\right) du
            \geq \int_0^s \int_0^\infty \Prob(T_{2n}\in dv)(1-\delta)\Prob(X > x-(K+\epsilon)v)du,
        \end{multline*}
        then evaluating the integrals and noticing the integrand is constant with respect to $u$ gives, for $x>R(s)$,
        \begin{align*}
            \int_0^s \int_0^\infty \Prob(T_{2n}\in dv) \Prob\left(X> x - (K+\epsilon)(u + v)\right) du
            \geq s(1-\delta) \Prob( X > x - (K+\epsilon)T_{2n} ).
        \end{align*}
        Now, consider some $l>0$ and suppose $x>R(l)$, so that $x>R(s)$ for any $s\in[0,l]$. Then,
        \begin{align*}
            &\int_0^\infty \Prob(T_2 \in ds)\int_0^s \int_0^\infty \Prob(T_{2n}\in dv)\Prob(X>x-(K+\epsilon)(u+v)) du\\
            &\qquad\geq \int_0^l \Prob(T_2 \in ds) s(1-\delta)\Prob(X>x-(K+\epsilon)T_{2n})\\
            &\qquad=(1-\delta)\Prob(X>x-(K+\epsilon)T_{2n})\E[T_2; T_2<l].
        \end{align*}
        Since $l>0$ and $\delta>0$ are arbitrary and $T_2$ is integrable, we can take $l$ sufficiently large to obtain $\E[T_2;T_2<l]\geq(1-\delta)\E[T_2]$ and so
        \begin{multline*}
            \int_0^\infty \Prob(T_2 \in ds)\int_0^s \int_0^\infty \Prob(T_{2n}\in dv)\Prob(X>x-(K+\epsilon)(u+v)) du
             \geq (1-\delta)^2\Prob(X>x-(K+\epsilon)T_{2n})\E[T_2].
        \end{multline*}
        If this is substituted into the expression for $G_X$ for $x>R(l)$ we have
        \begin{align*}
            G_X(x) \geq (1-\delta)^2|K+\epsilon|\E[T_2] \sum_{n=0}^\infty \Prob( X>x - (K+\epsilon)T_{2n} ).
        \end{align*}
        For the upper bound, since $-(K+\epsilon)u>0$ for $u>0$,
        \begin{align*}
            \E\left[\int_0^{\tilde{T}_2}\Prob\left( \vphantom{\Big(} X>x -(K+\epsilon)(u+T_{2n}) \:\middle|\: \sigma(T_{2n})\right) du\right]
            \leq \E[\tilde{T_2}]\Prob\left( X > x-(K+\epsilon)T_{2n}\right),
        \end{align*}
        which substituted into the expression for $G_X$ gives, for all $x>0$,
        \begin{align*}
            G_X(x) \leq \E[T_2]|K+\epsilon| \sum_{n=0}^\infty \Prob( X>x-(K+\epsilon)T_{2n} ).
        \end{align*}
        Combining the upper and lower bounds gives equation (\ref{eqn:tailSumDecomp}).

    \end{proof}

\subsection{Framework for an upper bound of \texorpdfstring{$\log(A_\infty)$}{log(A)}}
    We now develop a framework for bounding $\log(A_\infty)$ whenever $\E[\xi_{T_2}]\in(-\infty,0)$. This will be used in the strong subexponential setting of Section \ref{sec:strongSubexpType} to obtain the right tails of $A_\infty$.

    For $\epsilon\in(0,-K)$ and sufficiently large $A\in\R$ define a sequence of stopping times by $\sigma_0:=0$ and
    \begin{equation}
            \sigma_n := \inf\left\{ t>\sigma_{n-1} \::\: \xi_t-\xi_{\sigma_{n-1}} \geq (K+\epsilon)(t-\sigma_{n-1}) + A \right\},
    \end{equation}
    for each $n\in\N$, with the convention $\inf(\varnothing)=\infty$, and setting $\sigma_n=\infty$ if $\sigma_{n-1}=\infty$. Then also define
    \begin{align*}
        N:=\max\{ n\in\N_0 \:|\: \sigma_n<\infty \}
        \qquad\text{and}\qquad
        \rho_n&:=\Prob(\sigma_n<\infty \:|\: \sigma_{n-1}<\infty ).
    \end{align*}
    The next lemma concerns the finiteness of $N$.

    \begin{lem}
        \label{lem:Nfinite}
        If $\E[\xi_{T_2}]<0$, then there exists an $A^*>0$ such that for all $A>A^*$, $N$ is a.s. finite.
    \end{lem}
    \begin{proof}
        Define a new MAP $\{(J_t,\tilde{\xi_t})\::\:t\geq0\}$ by setting $\tilde{\xi}_t := \xi_t - (K+\epsilon) t$. For each $n\in\N$,
        \begin{align*}
            1-\rho_n
            &= \Prob\left( \sup_{t>\sigma_{n-1}}\tilde{\xi}_t-\tilde{\xi}_{\sigma_{n-1}} < A \:\middle|\: \sigma_{n-1}<\infty \right),
        \end{align*}
        then, since $\sigma_{n-1}$ is a stopping time, using the Markov additive property and summing over the events $\{J_{\sigma_{n-1}}=j\}$ for $j\in\{+,-\}$
        \begin{align*}
             1-\rho_n
             = \sum_{j\in\{+,-\}}\Prob_j\left( \sup_{t\geq0}\tilde{\xi}_t<A \right)\Prob(J_{\sigma_{n-1}} = j).
        \end{align*}
        By the strong law of large numbers,
        \begin{align*}
            \lim_{t\rightarrow\infty}t^{-1}\tilde{\xi_t}
            &= \lim_{t\rightarrow\infty}t^{-1}(\xi_t - (K+\epsilon)t)
            = K-(K+\epsilon)
            = -\epsilon <0
            \qquad\text{a.s.}
        \end{align*}
        and hence there a.s. exists $T>0$ such that if $t>T$ then $\tilde{\xi}_t<0$. From this we can conclude $\sup_{t\geq0}\tilde{\xi}_t = \max\left(0 \:,\: \sup_{t\in[0,T]}\tilde{\xi}_t\right)<\infty$ since the supremum of a \cadlag process over a compact interval is bounded.

        This implies that there exists an $A^*>0$ such that for all $A>A^*$ we have $\Prob\left(\sup_{t\geq0}\tilde{\xi}_t > A\right)<1$. From this, we conclude $N<\infty$ a.s. since
        \begin{align}
            \label{eqn:NGeomBound}
            \Prob(N>n)
            & = \prod_{k=1}^n\Prob(\sigma_{k}<\infty \:|\: \sigma_{k-1}<\infty )
            \leq  \max_{j\in\{+,-\}} \Prob_j\left( \sup_{t\geq0}\tilde{\xi}_t < A \right)^n.
        \end{align}
    \end{proof}

    Since there are conditions for $N$ to be finite, we can bound $\log(A_\infty)$ by using the stopping times $\{\sigma_n\}_{n\in\N}$ to split the process $\{\xi_t:t\geq0\}$ into a finite number of bounded sections.

    Define the constant
    \begin{align*}
        C:= \log\left( \frac{e^A}{|K+\epsilon|} \right)
    \end{align*}
    and by taking $A>A^*$ sufficiently large we can ensure $e^C>2$. Then, we have the following upper bound for $\log(A_\infty)$.%, which corresponds to the result in the \levy process case given in \cite[pp 11, Lemma 4.1]{MaulikZwart2006156}.
    \begin{lem}
        \label{lem:logAsum}
        If $\E[\xi_{T_2}]<0$, then
        \begin{equation*}
            \log{A_\infty} \leq (N+1)C + \sum_{n=1}^N(\xi_{\sigma_n} - \xi_{\sigma_{n-1}})^+,
        \end{equation*}
        where $(\cdot)^+$ denotes the positive part.
    \end{lem}
    \begin{proof}
        Following the approach of \cite[pp 11, Lemma 4.1]{MaulikZwart2006156},
        $A_\infty$ may be expanded as
        \begin{multline*}
            A_\infty = \int_0^{\sigma_1} e^{\xi_t} dt + e^{\xi_{\sigma_1}}\left( \int_{\sigma_1}^{\sigma_2}e^{\xi_t - \xi_{\sigma_1}}dt  + e^{\xi_{\sigma_2}-\xi_{\sigma_1}}\left( \int_{\sigma_2}^{\sigma_3}e^{\xi_t - \xi_{\sigma_2}}dt +  \dots \right.\right.\\
            \left.\left.\dots + e^{\xi_{\sigma_N}-\xi_{\sigma_{N-1}}}\left( \int_{\sigma_N}^{\sigma_{N+1}}e^{\xi_t - \xi_{\sigma_N}}dt \right) \right)\right),
        \end{multline*}
        noting that $\sigma_{N+1}=\infty$. By the definition of $\sigma_n$, we have
        \begin{equation*}
            \int_{\sigma_n}^{\sigma_{n+1}}e^{\xi_t - \xi_{\sigma_n}}dt
                \leq \int_{\sigma_n}^{\sigma_{n+1}}\exp( (K+\epsilon)(t-\sigma_n) + A) dt
                \leq e^C,
        \end{equation*}
        which substituted into the expression for $A_\infty$ gives
        \begin{align*}
            A_\infty
            \leq
            e^C + e^{\xi_{\sigma_1}}\left( e^C  + e^{\xi_{\sigma_2}-\xi_{\sigma_1}}\left( e^C +
            \dots + e^{\xi_{\sigma_N}-\xi_{\sigma_{N-1}}}\left( e^C \right)\right)\right).
        \end{align*}
        Then, by considering the logarithm of both sides of and repeatedly using the property $\log(A+B)\leq\log(A)+\log(B)$ whenever $A,B>2$,
        \begin{align*}
            \log(A_\infty)
            \leq(N+1)C + \sum_{n=1}^N(\xi_{\sigma_n} - \xi_{\sigma_{n-1}})^+
        \end{align*}
        where we have made use of the fact $e^{(\xi_{\sigma_n}-\xi_{\sigma_{n-1}})^+}\geq1$  and $e^C>2$ in order to use the log inequality.
    \end{proof}

    As a consequence of this lemma, the right tails of $\log(A_\infty)$ can be studied  by considering the evolution of the MAP between the stopping times $\{\sigma_n\}_{n\in\N}.$ First we consider the $J$ component.

    Let $(K_n)_{n\in\N_0}$ be the sequence of random variables, taking values in $\{+,-,\infty\}$, such that for each $n\in\N$, if $\sigma_n<\infty$ then $K_n=J_{\sigma_n}$ otherwise $K_n=\infty$.
    For each $\alpha,\beta\in\{+,-\}$, we will be interested in the number of times that $\{K_n\}_{n\in\N}$ transitions from $\alpha$ to $\beta$. For this purpose,  define the random variable $N(\alpha,\beta):=\sum_{k=1}^\infty\mathbbm{1}_{\{K_{k-1}=\alpha,K_k=\beta\}}$.
    We also make use of the notation $f=o(g)$ for any two functions $f,g:\R^+\rightarrow\R^+$ such that  $\lim_{x\rightarrow\infty}f(x)/g(x) = 0$.

    \begin{prop}
        \label{prop:markovChainK}
        Suppose $\E[\xi_{T_2}]\in(-\infty,0)$. Then, the sequence $(K_n)_{n\in\N_0}$ is a discrete     time homogeneous Markov chain, with $\infty$ as an absorbing state. Moreover, if $\eta$ is the stochastic matrix of $\{K_n\}_{n\in\N}$ and $\alpha,\beta,\gamma\in\{+,-\}$ with $\alpha\neq\beta$, then
        \begin{align*}
            \eta_{\alpha,\gamma}&\rightarrow0,
            \qquad
	   \E_\alpha[N(\alpha,\gamma)]\sim \eta_{\alpha,\gamma}
	   \qquad\text{and}\qquad
	  \E_\beta[N(\alpha,\gamma)] = o(\eta_{\alpha,\gamma})
        \end{align*}
        as $A\rightarrow\infty$.
        \if{}
        \begin{align}
            \eta_{\alpha,\beta}&\rightarrow0\\
            \E_\alpha[N(\alpha,\gamma)] &=
            \frac{\eta_{\alpha,\gamma}(1-\eta_{\beta,\beta})}{(1-\eta_{\beta,\beta})(1-\eta_{\alpha,\alpha}) - \eta_{\beta,\alpha}\eta_{\alpha,\beta}} \sim \eta_{\alpha,\gamma}\\
            \E_\beta[N(\alpha,\gamma)] &= \frac{\eta_{\beta,\alpha}\eta_{\alpha,\gamma}}{(1-\eta_{\alpha,\alpha})(1-\eta_{\beta,\beta}) - \eta_{\beta,\alpha}\eta_{\alpha,\beta}}=o(\eta_{\alpha,\gamma})
        \end{align}
        \fi
    \end{prop}
    \begin{proof}
        First, we show that $\{K_n\}_{n\in\N}$ is a Markov chain. If $K_{n-1}\neq\infty$ then by the Markov additive property, since $\sigma_{n-1}$ is a stopping time, $\{\xi_{\sigma_{n-1}+t}-\xi_{\sigma_{n-1}}\:|\: t\geq0\}$ is independent of $\mathcal{F}_{\sigma_{n-1}}$ given $K_{n-1}$. Moreover, the random variable
        \begin{align*}
            \Delta\sigma_n := \sigma_{n} - \sigma_{n-1}
            &= \inf\{ t\geq0 \::\: \xi_{t+\sigma_{n-1}} - \xi_{\sigma_{n-1}} \geq t(K+\epsilon) + A \},
        \end{align*}
        is a function of $\{\xi_{t+\sigma_{n-1}}-\xi_{\sigma_{n-1}}\::\:t\geq0\}$. Thus, the event $\{K_n=\infty\}=\{\Delta\sigma_n=\infty\}$ is independent of $\mathcal{F}_{\sigma_{n-1}}$ given $K_{n-1}$ and has the same law as the event $\{K_1=\infty\}$ given $K_0$.

        Moreover, if $\Delta\sigma_n<\infty$ then $K_n=J_{\sigma_n}=J_{\sigma_{n-1}+\Delta\sigma_n}$, hence $K_n$ is a function of\\ $\left\{ \left(\xi_{\sigma_{n-1}+t}-\xi_{\sigma_{n-1}},J_{\sigma_{n-1}+t}\right)\::\: t\geq0 \right\}$ and so by the Markov additive property is independent of $\mathcal{F}_{\sigma_{n-1}}$ given $K_{n-1}$ and has the same distribution as $K_1$ given $K_0$. Hence, the sequence $(K_n)_{n\in\N}$ is a time homogeneous Markov chain. By definition of $\sigma_n$, $\infty$ is clearly an absorbing state for $(K_n)_{n\in\N}$.

        Now, we consider the limiting behaviour of $\eta$ as $A\rightarrow\infty$. Let $\alpha,\gamma\in\{+,-\}$. From the proof of Lemma \ref{lem:Nfinite}, we know that $\sup_{t\geq0}\tilde{\xi}_t<\infty$ a.s. where $\tilde{\xi_t}:=\xi_t-(K+\epsilon)t$. Thus
        \begin{align*}
            \lim_{A\rightarrow\infty}\Prob\left( \sup_{t\geq0}\left\{\xi_t -(K+\epsilon)t \right\} > A \right)
            = \lim_{A\rightarrow\infty}\Prob\left( \sup_{t\geq0}\tilde{\xi}_t > A \right)
            = 0,
        \end{align*}
        however,
        \begin{align*}
            \eta_{\alpha,+}+\eta_{\alpha,-} = \Prob_\alpha\left(\sigma_1<\infty \right) = \Prob_\alpha\left( \sup_{t\geq0}\left\{\xi_t-(K+\epsilon)t\right\}>A\right)
        \end{align*}
        and, since $\eta_{\alpha,\gamma}$ is non-negative, this implies $\lim_{A\rightarrow\infty}\eta_{\alpha,\gamma} = 0$.

        Further assume that $\gamma\in\{+,-\}$ and $\alpha\neq\beta$. Then, it is easily seen that
        \begin{align*}
            \E_\theta[N(\alpha,\gamma)]
            &=\sum_{n=1}^\infty \Prob_\theta(K_n=\gamma\:|\:K_{n-1}=\alpha)\Prob_\theta(K_{n-1}=\alpha)
            =\eta_{\alpha,\gamma}\left( \mathbbm{1}_{\{\theta=\alpha\}} + \phi_\sigma(\alpha) \right),
        \end{align*}
        where $\phi_\theta(\alpha) := \sum_{n=1}^\infty \Prob_\theta(K_n=\alpha)$. Since $\infty$ is an absorbing state of the Markov Chain $(K_n)_{n\in\N}$ and $\alpha\neq\beta$, for each $n\in\N$,
        \begin{align*}
            \Prob_\theta(K_n=\alpha)
            &= \Prob_\theta( K_n=\alpha\:|\: K_{n-1}=\alpha)\Prob_\theta(K_{n-1}=\alpha) + \Prob_\theta( K_n=\alpha\:|\: K_{n-1}=\beta )\Prob_\theta(K_{n-1}=\beta)\\
            &= \eta_{\alpha,\alpha}\Prob_\theta(K_{n-1}=\alpha) + \eta_{\beta,\alpha}\Prob_\theta(K_{n-1}=\beta),
        \end{align*}
        then summing up over $n\in\N$ we have
        \begin{align*}
            \phi_\theta(\alpha) = \eta_{\alpha,\alpha}\left(  \mathbbm{1}_{\{\theta=\alpha\}} + \phi_\theta(\alpha) \right) + \eta_{\beta,\alpha}\left( \mathbbm{1}_{\{\theta=\beta\}} + \phi_\theta(\beta) \right),
        \end{align*}
        and by symmetry
        \begin{align*}
            \phi_\theta(\beta) = \eta_{\beta,\beta}\left(  \mathbbm{1}_{\{\theta=\beta\}} + \phi_\theta(\beta) \right) + \eta_{\alpha,\beta}\left( \mathbbm{1}_{\{\theta=\alpha\}} + \phi_\theta(\alpha) \right).
        \end{align*}
        Solving this system gives
        \begin{align*}
            \phi_\theta(\alpha) = \frac{\mathbbm{1}_{\{\theta=\alpha\}}\left(\eta_{\alpha,\alpha}(1-\eta_{\beta,\beta}) + \eta_{\beta,\alpha}\eta_{\alpha,\beta}\right) + \mathbbm{1}_{\{\theta=\beta\}}\eta_{\beta,\alpha} }{ (1-\eta_{\alpha,\alpha})(1-\eta_{\beta,\beta}) -\eta_{\beta,\alpha}\eta_{\alpha,\beta} },
        \end{align*}
        thus
        \begin{align*}
            \E_\theta\left[ N(\alpha,\gamma) \right]
            =\frac{\eta_{\alpha,\gamma}\left( \mathbbm{1}_{\{\theta=\alpha\}}(1-\eta_{\beta,\beta}) + \mathbbm{1}_{\{\theta=\beta\}}\eta_{\beta,\alpha} \right)}{(1-\eta_{\alpha,\alpha})(1-\eta_{\beta,\beta}) -\eta_{\beta,\alpha}\eta_{\alpha,\beta}},
        \end{align*}
        from which the two asymptotic results for $\E_\theta[N(\alpha,\gamma)]$ are a consequence of the limiting behaviour of $\eta$.
    \end{proof}

\if{}
    \begin{lem}
        \label{lem:KmarkovChain}
        Under Assumption (I), the sequence $(K_n)_{n\in\N_0}$ is a discrete time homogeneous Markov chain with $\infty$ being an absorbing state.
    \end{lem}
    \begin{proof}
        If $K_{n-1}\neq\infty$ then by the Markov additive property, since $\sigma_{n-1}$ is a stopping time, $\{\xi_{\sigma_{n-1}+t}-\xi_{\sigma_{n-1}}\:|\: t\geq0\}$ is independent of $\mathcal{F}_{\sigma_{n-1}}$ given $K_{n-1}$. Moreover, the random variable
        \begin{align*}
            \Delta\sigma_n := \sigma_{n} - \sigma_{n-1}
            &= \inf\{ t\geq0 \::\: \xi_{t+\sigma_{n-1}} - \xi_{\sigma_{n-1}} \geq t(K+\epsilon) + A \},
        \end{align*}
        is a function of $\{\xi_{t+\sigma_{n-1}}-\xi_{\sigma_{n-1}}\::\:t\geq0\}$. Thus, the event $\{K_n=\infty\}=\{\Delta\sigma_n=\infty\}$ is independent of $\mathcal{F}_{\sigma_{n-1}}$ given $K_{n-1}$ and has the same law as the event $\{K_1=\infty\}$ given $K_0$.

        Moreover, if $\Delta\sigma_n<\infty$ then $K_n=J_{\sigma_n}=J_{\sigma_{n-1}+\Delta\sigma_n}$, hence $K_n$ is a function of\\ $\left\{ \left(\xi_{\sigma_{n-1}+t}-\xi_{\sigma_{n-1}},J_{\sigma_{n-1}+t}\right)\::\: t\geq0 \right\}$ and so by the Markov additive property is independent of $\mathcal{F}_{\sigma_{n-1}}$ given $K_{n-1}$ and has the same distribution as $K_1$ given $K_0$. Hence, the sequence $(K_n)_{n\in\N}$ is a time homogeneous Markov chain. By definition of $\sigma_n$, $\infty$ is clearly an absorbing state for $(K_n)_{n\in\N}$.

    \end{proof}

    Let $\eta$ denote the stochastic matrix of $(K_n)_{n\in\N}.$

    \begin{lem}
        \label{lem:etaLimit}
        Under Assumption (I), if $\alpha,\beta\in\{+,-\}$ then $\eta_{\alpha,\beta}\rightarrow0$ as $A\rightarrow\infty$.
    \end{lem}
    \begin{proof}
        Let $\alpha,\beta\in\{+,-\}$. From the proof of Lemma \ref{lem:Nfinite} we know that $\sup_{t\geq0}\tilde{\xi}_t<\infty$ a.s. where $\tilde{\xi_t}:=\xi_t-(K+\epsilon)t$. Thus
        \begin{align*}
            \lim_{A\rightarrow\infty}\Prob\left( \sup_{t\geq0}\left\{\xi_t -(K+\epsilon)t \right\} > A \right)
            = \lim_{A\rightarrow\infty}\Prob\left( \sup_{t\geq0}\tilde{\xi}_t > A \right)
            = 0,
        \end{align*}
        however,
        \begin{align*}
            \eta_{\alpha,+}+\eta_{\alpha,-} = \Prob_\alpha\left(\sigma_1<\infty \right) = \Prob_\alpha\left( \sup_{t\geq0}\left\{\xi_t-(K+\epsilon)t\right\}>A\right)
        \end{align*}
        and, since $\eta_{\alpha,\beta}$ is non-negative, this implies $\lim_{A\rightarrow\infty}\eta_{\alpha,\beta} = 0$.

    \end{proof}

    For each $\alpha,\beta\in\{+,-\}$ define the random variable $N(\alpha,\beta):=\sum_{k=1}^\infty\mathbbm{1}_{\{K_{k-1}=\alpha,K_k=\beta\}}$.
    \begin{lem}
        \label{lem:expN}
        Under Assumption (I), if $\alpha,\beta,\gamma\in\{+,-\}$ and $\alpha\neq\beta$ then
        \begin{align*}
            &\E_\alpha[N(\alpha,\gamma)] =
            \frac{\eta_{\alpha,\gamma}(1-\eta_{\beta,\beta})}{(1-\eta_{\beta,\beta})(1-\eta_{\alpha,\alpha}) - \eta_{\beta,\alpha}\eta_{\alpha,\beta}} \sim \eta_{\alpha,\gamma}
        \end{align*}
        and
        \begin{align*}
            \E_\beta[N(\alpha,\gamma)] &= \frac{\eta_{\beta,\alpha}\eta_{\alpha,\gamma}}{(1-\eta_{\alpha,\alpha})(1-\eta_{\beta,\beta}) - \eta_{\beta,\alpha}\eta_{\alpha,\beta}}=o(\eta_{\alpha,\gamma})
        \end{align*}
         as $A\rightarrow\infty$.
    \end{lem}

    \begin{proof}
        Suppose $\sigma,\alpha,\gamma\in\{+,-\}$. Then, it is easily seen that
        \begin{align*}
            \E_\sigma[N(\alpha,\gamma)]
            &=\sum_{n=1}^\infty \Prob_\sigma(K_n=\gamma\:|\:K_{n-1}=\alpha)\Prob_\sigma(K_{n-1}=\alpha)
            =\eta_{\alpha,\gamma}\left( \mathbbm{1}_{\{\sigma=\alpha\}} + \phi_\sigma(\alpha) \right),
        \end{align*}
        where $\phi_\sigma(\alpha) := \sum_{n=1}^\infty \Prob_\sigma(K_n=\alpha)$. Since $\infty$ is an absorbing state of the Markov Chain $(K_n)_{n\in\N}$ and $\alpha\neq\beta$, for each $n\in\N$,
        \begin{align*}
            \Prob_\sigma(K_n=\alpha)
            &= \Prob_\sigma( K_n=\alpha\:|\: K_{n-1}=\alpha)\Prob_\sigma(K_{n-1}=\alpha) + \Prob_\sigma( K_n=\alpha\:|\: K_{n-1}=\beta )\Prob_\sigma(K_{n-1}=\beta)\\
            &= \eta_{\alpha,\alpha}\Prob_\sigma(K_{n-1}=\alpha) + \eta_{\beta,\alpha}\Prob_\sigma(K_{n-1}=\beta),
        \end{align*}
        then summing up over $n\in\N$ we have
        \begin{align*}
            \phi_\sigma(\alpha) = \eta_{\alpha,\alpha}\left(  \mathbbm{1}_{\{\sigma=\alpha\}} + \phi_\sigma(\alpha) \right) + \eta_{\beta,\alpha}\left( \mathbbm{1}_{\{\sigma=\beta\}} + \phi_\sigma(\beta) \right),
        \end{align*}
        and by symmetry
        \begin{align*}
            \phi_\sigma(\beta) = \eta_{\beta,\beta}\left(  \mathbbm{1}_{\{\sigma=\beta\}} + \phi_\sigma(\beta) \right) + \eta_{\alpha,\beta}\left( \mathbbm{1}_{\{\sigma=\alpha\}} + \phi_\sigma(\alpha) \right).
        \end{align*}
        Solving this system gives
        \begin{align*}
            \phi_\sigma(\alpha) = \frac{\mathbbm{1}_{\{\sigma=\alpha\}}\left(\eta_{\alpha,\alpha}(1-\eta_{\beta,\beta}) + \eta_{\beta,\alpha}\eta_{\alpha,\beta}\right) + \mathbbm{1}_{\{\sigma=\beta\}}\eta_{\beta,\alpha} }{ (1-\eta_{\alpha,\alpha})(1-\eta_{\beta,\beta}) -\eta_{\beta,\alpha}\eta_{\alpha,\beta} },
        \end{align*}
        thus
        \begin{align*}
            \E_\sigma\left[ N(\alpha,\gamma) \right]
            =\frac{\eta_{\alpha,\gamma}\left( \mathbbm{1}_{\{\sigma=\alpha\}}(1-\eta_{\beta,\beta}) + \mathbbm{1}_{\{\sigma=\beta\}}\eta_{\beta,\alpha} \right)}{(1-\eta_{\alpha,\alpha})(1-\eta_{\beta,\beta}) -\eta_{\beta,\alpha}\eta_{\alpha,\beta}},
        \end{align*}
        from which the result of the Lemma is immediate. The asymptotic results then follow from Lemma \ref{lem:etaLimit}.

    \end{proof}
\fi

    In the next lemma we consider the evolution of $\xi$ between the stopping times $\{\sigma_n\}_{n\in\N}$, conditioned on the values of $\{K_n\}_{n\in\N}$.
    \begin{lem}
        \label{lem:xiIndp}
       Suppose that $\E[\xi_{T_2}]\in(-\infty,0)$ and $m,n\in\N$ with $m<n$. Then, conditionally on $K_{m-1}$, $K_m$, $K_{n-1}$ and $K_n$, the increments $\xi_{\sigma_m}-\xi_{\sigma_{m-1}}$ and $\xi_{\sigma_n} - \xi_{\sigma_{n-1}}$ are independent. If $\alpha,\beta\in\{+,-\}$, then conditional on the event $\{K_{n-1}=K_{m-1}=\alpha; K_n=K_m=\beta\}$ the increments $\xi_{\sigma_m}-\xi_{\sigma_{m-1}}$ and $\xi_{\sigma_n} - \xi_{\sigma_{n-1}}$ are equal in distribution and independent. Furthermore, we have for any $l\in\N$ such that $m\neq l$ and any bounded continuous function $f:\R\rightarrow\R^+$,
        \begin{align*}
            \E[ f(\xi_{\sigma_{m}}-\xi_{\sigma_{m-1}}) \:|\: \sigma(K_{m-1},K_m,K_{l-1},K_l) ]
            &=\E[ f(\xi_{\sigma_{m}}-\xi_{\sigma_{m-1}}) \:|\: \sigma(K_{m-1},K_m) ].
        \end{align*}
    \end{lem}
    \begin{proof}
        First suppose that $m<l$. Then, we have
        \begin{align*}
            &\E[f\left(\xi_{\sigma_m}-\xi_{\sigma_{m-1}}\right)\:|\: \sigma(K_{m-1},K_m,K_{l-1},K_l)]\\
            &\qquad =\sum_{\gamma,\delta\in\{+,-\}} \mathbbm{1}_{\{K_{l-1}=\gamma,K_{l}=\delta\}}\frac{\E[f\left(\xi_{\sigma_m}-\xi_{\sigma_{m-1}}\right); K_{l-1}=\gamma,K_l=\delta \: | \:\sigma(K_{m-1},K_m)]}{\Prob(K_{l-1}=\gamma,K_l=\delta \:|\: \sigma(K_{m-1},K_m))}.
        \end{align*}
        It follows that, by using the tower property and the fact that $(K_k)_{k\in\N}$ is a Markov chain, we have
        \begin{align*}
            &\E[f\left(\xi_{\sigma_m}-\xi_{\sigma_{m-1}}\right); K_{l-1}=\gamma,K_l=\delta \: | \:\sigma(K_{m-1},K_m)]\\
            &\qquad=\E\left[ f\left(\xi_{\sigma_m}-\xi_{\sigma_{m-1}}\right) \E\left[ \mathbbm{1}_{\{K_{l-1}=\gamma,K_l=\delta\}} \:|\: \mathcal{F}_{\sigma_m}\right] \:|\:\sigma(K_{m-1},K_m)  \right]\\
            &\qquad=\E\left[ f\left(\xi_{\sigma_m}-\xi_{\sigma_{m-1}}\right) \E\left[ \mathbbm{1}_{\{K_{l-1}=\gamma,K_l=\delta\}} \:|\:\sigma(K_m) \right] \:|\:\sigma(K_{m-1},K_m)  \right]\\
            &\qquad= \Prob\left(K_{l-1}=\gamma,K_l=\delta \:|\:\sigma(K_{m-1},K_m) \right) \E\left[ f\left(\xi_{\sigma_m}-\xi_{\sigma_{m-1}}\right) \:|\:\sigma(K_{m-1},K_m) \right],
        \end{align*}
        which, when substituted into the previous equation, gives
        \begin{align*}
            &\E[f\left(\xi_{\sigma_m}-\xi_{\sigma_{m-1}}\right)\:|\: \sigma(K_{m-1},K_m,K_{l-1},K_l)]\\
            &\qquad= \sum_{\gamma,\delta\in\{+,-\}} \mathbbm{1}_{\{K_{l-1}=\gamma,K_l=\delta\}}\E\left[ f\left( \xi_{\sigma_m}-\xi_{\sigma_{m-1}} \right) \:|\:\sigma(K_{m-1},K_m) \right]\\
            &\qquad=\E\left[ f(\xi_{\sigma_m}-\xi_{\sigma_{m-1}} ) \:|\: \sigma(K_{m-1},K_m) \right].
        \end{align*}
        Now, suppose $m>l$ and through a direct application of the Markov additive property we have
        \begin{align*}
            &\E\left[ f(\xi_{\sigma_{m}}-\xi_{\sigma_{m-1}}) \:|\: \sigma(K_{l-1},K_l,K_{m-1},K_m) \right]\\
            &\qquad=\sum_{\alpha\in\{+,-\}}\mathbbm{1}_{\{K_m=\alpha\}}\frac{\E\left[ f(\xi_{\sigma_{m}}-\xi_{\sigma_{m-1}}); K_m=\alpha \:|\: \sigma(K_{l-1},K_l,K_{m-1}) \right]}{\Prob(K_m=\alpha\:|\: \sigma(K_{l-1},K_l,K_{m-1})  ) }\\
            &\qquad=\sum_{\alpha\in\{+,-\}}\mathbbm{1}_{\{K_m=\alpha\}} \frac{\E\left[ f(\xi_{\sigma_{m}}-\xi_{\sigma_{m-1}}) ; K_m=\alpha \:|\: \sigma( K_{m-1} ) \right]}{\Prob(K_m=\alpha\:|\: \sigma (K_{m-1})  )}\\
            &\qquad= \E\left[ f(\xi_{\sigma_{m}}-\xi_{\sigma_{m-1}})  \:|\: \sigma(K_{m-1},K_m) \right].
        \end{align*}

        To see the independence of increments, suppose that $f,g:\R\rightarrow\R^+$ are bounded continuous functions, then
        \begin{align*}
            &\E\left[ f(\xi_{\sigma_n}-\xi_{\sigma_{n-1}})g(\xi_{\sigma_m} - \xi_{\sigma_{m-1}}) \:|\: \sigma(K_{m-1},K_m,K_{n-1},K_n)  \right]\\
            &\qquad= \sum_{\alpha\in\{+,-\}}\mathbbm{1}_{\{K_n=\alpha\}}\frac{\E\left[ f(\xi_{\sigma_n}-\xi_{\sigma_{n-1}})g(\xi_{\sigma_m} - \xi_{\sigma_{m-1}}); K_n=\alpha \:|\: \sigma(K_{m-1},K_m,K_{n-1})  \right]}{\Prob\left( K_n=\alpha \:|\: \sigma(K_{m-1},K_m,K_{n-1}) \right)}.
        \end{align*}
        Then, by the tower property, we get
        \begin{align*}
            &\E\left[ f(\xi_{\sigma_n}-\xi_{\sigma_{n-1}})g(\xi_{\sigma_m} - \xi_{\sigma_{m-1}}); K_n=\alpha \:|\: \sigma(K_{m-1},K_m,K_{n-1})  \right]\\
            &\qquad=\E\left[ g(\xi_{\sigma_m} - \xi_{\sigma_{m-1}}) \E\left[  f(\xi_{\sigma_n}-\xi_{\sigma_{n-1}}); K_n=\alpha\:|\:\mathcal{F}_{\sigma_{n-1}} \right] \:|\: \sigma(K_{m-1},K_m,K_{n-1})  \right]\\
            &\qquad=\E\left[ g(\xi_{\sigma_m} - \xi_{\sigma_{m-1}}) \E\left[  f(\xi_{\sigma_n}-\xi_{\sigma_{n-1}}); K_n=\alpha\:|\:\sigma(K_{\sigma_{n-1}}) \right] \:|\: \sigma(K_{m-1},K_m,K_{n-1})  \right]\\
            &\qquad=\E\left[  f(\xi_{\sigma_n}-\xi_{\sigma_{n-1}}); K_n=\alpha\:|\:\sigma(K_{\sigma_{n-1}}) \right]\E\left[ g(\xi_{\sigma_m} - \xi_{\sigma_{m-1}})  \:|\: \sigma(K_{m-1},K_m)  \right].
        \end{align*}
        Plugging this into the previous equation yields
        \begin{align*}
            &\E\left[ f(\xi_{\sigma_n}-\xi_{\sigma_{n-1}})g(\xi_{\sigma_m} - \xi_{\sigma_{m-1}}) \:|\: \sigma(K_{m-1},K_m,K_{n-1},K_n)  \right]\\
            &\qquad= \sum_{\alpha\in\{+,-\}}\mathbbm{1}_{\{K_n=\alpha\}}\frac{\E\left[  f(\xi_{\sigma_n}-\xi_{\sigma_{n-1}}); K_n=\alpha\:|\:\sigma(K_{\sigma_{n-1}}) \right]\E\left[ g(\xi_{\sigma_m} - \xi_{\sigma_{m-1}})  \:|\: \sigma(K_{m-1},K_m)  \right]}{\Prob\left( K_n=\alpha \:|\: \sigma(K_{m-1},K_m,K_{n-1}) \right)}\\
            &\qquad= \E\left[g(\xi_{\sigma_m} - \xi_{\sigma_{m-1}})  \:|\: \sigma(K_{m-1},K_m)  \right]\sum_{\alpha\in\{+,-\}}\mathbbm{1}_{\{K_n=\alpha\}}\frac{\E\left[  f(\xi_{\sigma_n}-\xi_{\sigma_{n-1}}); K_n=\alpha\:|\:\sigma(K_{\sigma_{n-1}}) \right] }{\Prob\left( K_n=\alpha \:|\: \sigma(K_{n-1}) \right)}\\
            &\qquad= \E\left[g(\xi_{\sigma_m} - \xi_{\sigma_{m-1}})  \:|\: \sigma(K_{m-1},K_m)  \right]\E\left[  f(\xi_{\sigma_n}-\xi_{\sigma_{n-1}}) \:|\:\sigma(K_{\sigma_{n-1}},K_{\sigma_n}) \right].
        \end{align*}
    \end{proof}

\subsection{Lamperti-Kiu processes of strong subexponential type}
    \label{sec:strongSubexpType}
    Strong subexponential distributions are a widely studied class of heavy tailed dsitributions, both because of their mathematical tractability and their appearance in empirical data. We will use \cite{foss2013introduction} as a reference to the background theory of subexponential distributions, within which further discussion of the use of these distributions can be found.

    For a probability distribution $Q:\R^+\rightarrow[0,1]$ define $\bar{Q}(x)\coloneqq 1-Q(x)$ for all $x\in\R^+$. Then, if $\overline{Q*Q}(x)/\overline{Q}(x)\rightarrow 2$ as $x\rightarrow\infty$, we say that $Q$ is a subexponential distribution. It can be shown (for instance see \cite{foss2013introduction}) that all subexponential distributions are also long-tailed.  The distribution $Q$ is a strong subexponential distribution if it also satisfies the property that
\begin{align*}
	\lim_{x\rightarrow\infty}\frac{1}{\bar{Q}(x)}\int_0^x\bar{Q}(x-y)\bar{Q}(y) dy = 2m,
\end{align*}
where $m=\E[X^+]$ and $X$ is a random variable with distribution $Q$. We will refer to a random variable with a (strong) subexponential distribution as a (strong) subexponential random variable.

    Let $\mathcal{S}$ denote the set of real valued subexponential random variables and $\mathcal{S}^*$ denote the subset of $\mathcal{S}$ comprising of strong subexponential random variables.

    For a random variable $X$, recall the definitions:
    \begin{align}
        \label{eqn:defnIntegratedTail}
        G_X(x) := \int_x^\infty \Prob(X>u) du
        \qquad\text{and}\qquad
        H_X(x) := \min\left( 1, G_X(x) \right).
    \end{align}

    If $H_X$ is a subexponential distribution then we write $X\in\mathcal{S}_I$ and from \cite[Chapter 3, pp 55, Theorem 3.27]{foss2013introduction}, we have $\mathcal{S}^*\subset\mathcal{S}_I$.

    For ease of notation, we also define the integrated tails $H(x) := H_{\xi_{T_2}}(x)$ and, for each $\alpha\in\{+,-\}$, define $H_{\xi_\alpha}:=H_{\xi^{(\alpha)}_{\zeta_\alpha}}$ and $H^{(\alpha)} := H_{\xi_\alpha} + H_{U_{-\alpha}}$.  Let us introduce a subset of components of the Lamperti-Kiu decomposition given by $L:=\{ \xi^{(+)}_{\zeta_+} , \xi^{(-)}_{\zeta_-} , U^+, U^- \}$.

    For any two functions $f,g:\R^+\rightarrow\R^+$, we will write $f=\mathcal{O}(g)$ if $\limsup_{x\rightarrow\infty}f(x)/g(x)\in\R$ .
    \begin{defn}
        We will say that a Lamperti-Kiu process is of \textit{strong subexponential type} if $\xi_{T_2}$ is long tailed and there exists $X\in L$ such that $X\in\mathcal{S}^*$ and for all $W\in L\setminus\{X\}$ we have  $\Prob(W>x)=\mathcal{O}(\Prob(X>x))$.
    \end{defn}
    If a Lamperti-Kiu process is of this type, there is a heaviest tailed component of the Lamperti-Kiu decomposition and it is strong subexponential. Denote this component $X\in L$.    Let $B\subseteq\{+,-\}$ be the set of all $\beta\in\{+,-\}$ such that $\limsup_{x\rightarrow\infty}H_X(x)^{-1}H^{(\beta)}(x)\neq0$. Then, for any $b\in B$ and $\beta\in\{+,-\}\setminus B$, we have that, $H^{(\beta)}(x) = o\left( H^{(b)}(x) \right)$ as $x\rightarrow\infty$.

    By the closure properties of $\mathcal{S}^*$ \cite[pp 52, Chapter 3, Corollary 3.16]{foss2013introduction}, it follows that $\xi_{T_2}$ is also strong subexponential with tails and integrated tails respectively given, as $x\rightarrow\infty$, by
    \begin{align*}
        \Prob\left(\xi_{T_2} > x\right)
        \sim \sum_{\beta\in\{+,-\}}\left(\Prob\left( \xi^{(\beta)}_{\zeta_\beta} > x\right) + \Prob\left( U^\beta > x\right) \right)
    \end{align*}
    and
    \begin{align*}
        H(x)
        \sim \sum_{\beta\in\{+,-\}} H^{(\beta)}(x)
        \sim \sum_{\beta\in B} H^{(\beta)}(x).
    \end{align*}

    Recall from Section \ref{sec:finiteness} that $K:=\E[\xi_{T_2}]/\E[T_2]$ hence, when $\E[\xi_{T_2}]\neq0$, it follows that $K\neq0$. The main result of this section, which extends \cite[Section 4, pp 166]{MaulikZwart2006156} to Lamperti-Kiu processes, is the following result.

    \begin{thm}
        \label{thm:subexponTails}
        Suppose that $Y$ is a Lamperti-Kiu process of strong subexponential type such that $\E[\xi_{T_2}]\in(-\infty,0)$. Then,
        \begin{align}
            \label{eqn:expTails}
            \Prob(A_\infty > x) \sim \frac{H(\log(x))}{\E[T_2]|K|},
            \qquad\text{ as }x\rightarrow\infty.
        \end{align}
        Furthermore, $A_\infty$ is long tailed and $\log(A_\infty)$ is subexponential.
    \end{thm}

    Set $Z_n := \xi_{\sigma_n}-\xi_{\sigma_{n-1}}$ for each $n\geq1$. We are now in a position to consider the asymptotic behaviour of the survival function of $Z_n$ conditioned on $K_{n-1}$ and $K_n$, under the assumptions of Theorem \ref{thm:subexponTails}. Recall that for $\alpha\in\{+,-\}$, $H^{(\alpha)} = H_{\xi_\alpha}(x) + H_{U_{-\alpha}}(x)$.
    \begin{lem}
        \label{lem:limZalphaBeta}
        Suppose that $Y$ is a Lamperti-Kiu process of strong subexponential type such that $\E[\xi_{T_2}]\in(-\infty,0)$ and fix $\alpha,\beta\in\{+,-\}$. Then, if $\beta\in B$, for each $n\in\N$
        \begin{align*}
            \limsup_{x\rightarrow\infty}\frac{\Prob(Z_n >x\:|\: K_{n-1}=\alpha,K_n=\beta) }{H^{(\beta)}(x)} \leq \frac{1}{\eta^{(\alpha,\beta)}|K+\epsilon|\E[T_2]}.
        \end{align*}
        Furthermore, if $\beta\in\{+,-\}\setminus B$ and $b\in B$ then, for each $n\in\N$
        \begin{align*}
            \limsup_{x\rightarrow\infty}\frac{\Prob(Z_n >x \:|\: K_{n-1}=\alpha,K_n=\beta) }{H^{(b)}(x)} = 0.
        \end{align*}
    \end{lem}

    \begin{proof}
        Suppose that $x>A$, let $u_0\in(0,x-A)$ and fix $\alpha,\beta\in\{+,-\}$. For ease of notation, let $\sigma:=\sigma_1$. Then, since $\Prob(Z_n >x\:|\:K_{n-1}=\alpha,K_n=\beta ) = \Prob_\alpha(\xi_{\sigma} > x \:|\: \sigma< \infty, \: J_{\sigma}=\beta)$, we have
        \begin{align*}
            \Prob(Z_n >x\:|\:K_{n-1}=\alpha,K_n=\beta )
            = \frac{1}{\eta^{(\alpha,\beta)}}\sum_{m=0}^\infty \Prob_\alpha\left(\xi_{\sigma}>x;\: T_{m}\leq \sigma < T_{m+1} ;\: J_\sigma=\beta \right).
        \end{align*}
        To bound the elements of the sum first consider the strict inequality $T_{m}<\sigma<T_{m+1}$ for some $m\in\N$, then,
        \begin{align*}
            \Prob_\alpha\left( \xi_\sigma>x;\: T_{m}<\sigma<T_{m+1}; J_\sigma=\beta \right)
            \leq  \Prob_\alpha\left(\sup_{T_{m}< u<T_{m+1}} \xi_u > x ;\: \xi_{T_{m}}<(K+\epsilon)T_{m}+A; J_{T_m}=\beta \right)
        \end{align*}
        and using the Lamperti-Kiu decomposition followed by Lemma \ref{lem:WillekensExtension2} we have
        \begin{align*}
            \Prob_\alpha\left( \xi_\sigma>x;\: T_{m}<\sigma<T_{m+1}; J_\sigma=\beta \right)
            &\leq  \Prob_\alpha\left( \sup_{0< u<\tilde{\zeta}_\beta} \tilde{\xi}^{(\beta)}_u > x - (K+\epsilon)T_m - A; J_{T_m}=\beta \right)\\
            &\leq \frac{\Prob_\alpha\left( \tilde{\xi}^{(\beta)}_{\tilde{\zeta}_\beta} \geq x-(K+\epsilon)T_m-A-u_0; J_{T_m}=\beta \right)}{ \Prob\left( \xi^{(\beta)}_{\zeta_\beta} \geq -u_0 \right) }
        \end{align*}
        where $\tilde{\xi}^{(\beta)}$ and $\tilde{\zeta}_\beta$ are independent copies of the \levy process $\xi^{(\beta)}$ and the exponential random variable $\zeta_\beta$, respectively.

        In the case that $\sigma=T_{m}$,
        \begin{align*}
            \Prob_\alpha\left(\xi_\sigma>x;\:T_{m}=\sigma; J_\sigma=\beta \right)
            &\leq\Prob_\alpha\left( \xi_{T_{m}}>x;\: \xi_{T_m-} \leq (K+\epsilon)T_{m}+A;\:  J_{T_m}=\beta \right)\\
            &\leq\Prob_\alpha\left( \xi_{T_{m}} - \xi_{T_m-} > x - (K+\epsilon)T_{m}-A; J_{T_m}=\beta \right)
            \\&= \Prob_\alpha\left( U^{(-\beta)}>x-(K+\epsilon)T_{m} - A; J_{T_m}=\beta \right).
        \end{align*}

        If $\alpha=\beta$ then there must be an even number of changes of $J$ before $\sigma$, so there exists $m\in\N$ such that $\sigma\in[T_{2m},T_{2m+1})$. Hence combining the two results above gives,
        \begin{align*}
            \Prob\left( Z_n >x\:|\:K_{n-1}=\alpha,K_n=\beta \right)
            \leq& \frac{1}{\eta^{(\alpha,\beta)}} \left(\sum_{m=0}^\infty  \frac{\Prob_\alpha\left(\tilde{\xi}_{\tilde{\zeta}_\beta}^{(\beta)} \geq x-(K+\epsilon)T_{2m}-A-u_0\right)}{\Prob\left(\xi^{(\beta)}_{\zeta_\beta}>-u_0\right)}\right. \\
            &\left. + \sum_{m=0}^\infty\Prob_\alpha\left( U^{(-\beta)} \geq x- (K+\epsilon)T_{2m} - A -u_0\right) \right).
        \end{align*}

        If $\alpha\neq\beta$ then there is an odd number of changes in $J$ before time $\sigma$. However, $T_{2m+1}\geq T_{2m}$ so the inequalities can be weakened to give the same result as the $\alpha=\beta$ case.

        For ease of notation define
        \begin{align*}
            Q_\beta(u_0) := \Prob\left( \xi_{\zeta_\beta}^{(\beta)} \geq -u_0 \right)\leq 1.
        \end{align*}

        If $\xi^{(\beta)}_{\zeta_\beta}$ is long tailed  we can use Lemma \ref{lem:tailsumdecomp} to obtain the asymptotic approximation
        \begin{align*}
             \sum_{m=0}^\infty \Prob_\alpha\left( \tilde{\xi}^{(\beta)}_{\tilde{\zeta}_\beta} \geq x-(K+\epsilon)T_{2m}-A-u_0 \right) &\sim \frac{G_{\xi_\beta}(x)}{|K+\epsilon|\E[T_2]},
        \end{align*}
         as $x\rightarrow\infty$. Similarly, if $U^{(-\beta)}$ is long tailed we have
        \begin{align*}
            \sum_{m=0}^\infty\Prob_{\alpha}\left( U^{(-\beta)}> x-(K+\epsilon)T_{2m} -A -u_0 \right) &\sim \frac{G_{U^{(-\beta)}}(x)}{|K+\epsilon|\E[T_2]},
        \end{align*}
        as $x\rightarrow\infty$. We will consider separately the cases where both of the asymptotics hold, exactly one holds or neither hold.

        In the case where both $\xi^{(\beta)}_{\zeta_\beta}$ and $U^{(-\beta)}$ are subexponential (and hence are long-tailed) for all $\delta>0$, there exists an $R>0$ such that for all $x>R$,
        \begin{align*}
            \frac{\Prob\left(Z^{(\alpha,\beta)} > x\right)}{G_{\xi_\beta}(x) + G_{U^{(-\beta)}}(x)}
            &\leq \frac{1}{\left(G_{\xi_\beta}(x)+G_{U^{(-\beta)}}(x)\right)\eta^{(\alpha,\beta)}}\left(\frac{(1+\delta)G_{\xi_\beta}(x)}{Q_\beta(u_0)|K+\epsilon|\E[T_2]} + \frac{(1+\delta)G_{U^{(-\beta)}}(x)}{|K+\epsilon|\E[T_2]} \right)\\
            &\leq \frac{(1+\delta)}{\eta^{(\alpha,\beta)}Q_\beta(u_0)|K+\epsilon|\E[T_2]}
        \end{align*}
        where the second inequality holds since $Q_\beta(u_0)<1$. Since $\delta$ was arbitrary, taking the $\limsup$ as $x\rightarrow\infty$ yields
        \begin{align*}
            \limsup_{x\rightarrow\infty} \frac{\Prob\left(Z_n >x\:|\:K_{n-1}=\alpha,K_n=\beta\right)}{G_{\xi_\beta}(x) + G_{U^{(-\beta)}}(x)} \leq \frac{1}{\eta^{(\alpha,\beta)}Q_\beta(u_0)|K+\epsilon|\E[T_2]}.
        \end{align*}

        In the case where exactly one of $\xi^{(\beta)}_{\zeta_\beta}$ and $U_{-\beta}$ is subexponential, it asymptotically dominates the other as $x\rightarrow\infty$, since $Y$ is of strong subexponential type. Suppose that it is $\xi^{(\beta)}_{\zeta_\beta}$ that is subexponential and note that the following argument is symmetric in $\xi^{(\beta)}_{\zeta_\beta}$ and $U^{(-\beta)}$. For all $\delta>0$ there exists $\hat{\delta}>0$ such that $\hat{\delta}(1+\hat{\delta})<\delta/2$ and  an $R>0$ such that for all $x>R$ and $n\in\N$,
        \begin{align*}
            \Prob_{\alpha}\left( U^{(-\beta)}> x-(K+\epsilon)T_{2n} -A -u_0\right)
            \leq \hat{\delta}\Prob_\alpha\left( \tilde{\xi}^{(\beta)}_{\tilde{\zeta}_\beta} \geq x-(K+\epsilon)T_{2n}-A-u_0 \right).
        \end{align*}
        Thus, for all $x>R$ for suitably large $R$,
        \begin{align*}
            \sum_{m=0}^\infty \Prob_{\alpha}\left( U^{(-\beta)}> x-(K+\epsilon)T_{2m} -A -u_0\right)
            \leq \frac{\hat{\delta}(1+\hat{\delta})G_{\xi_\beta}(x)}{|K+\epsilon|\E[T_2]},
        \end{align*}
        which gives for $x>R$,
        \begin{align*}
            \frac{\Prob\left(Z_n >x\:|\:K_{n-1}=\alpha,K_n=\beta\right)}{G_{\xi_\beta}(x) + G_{U^{(-\beta)}}(x)}
            &\leq \frac{1}{\eta^{(\alpha,\beta)}G_{\xi_\beta}(x)}\left(\frac{(1+\frac{\delta}{2})G_{\xi_\beta}(x)}{|K+\epsilon|\E[T_2]Q_\beta(u_0)} + \frac{\frac{\delta}{2}G_{\xi_\beta}}{|K+\epsilon|\E[T_2]}  \right)\\
            &\leq \frac{1+\delta}{\eta^{(\alpha,\beta)}|K+\epsilon|\E[T_2]Q_\beta(u_0)}.
        \end{align*}
        And so, since $\delta>0$ was arbitrary, we may take the $\limsup$ first as $\delta\rightarrow0$ and then as $x\rightarrow\infty$ to obtain
        \begin{align*}
            \limsup_{x\rightarrow\infty} \frac{\Prob\left(Z_n >x\:|\:K_{n-1}=\alpha,K_n=\beta\right)}{G_{\xi_\beta}(x) + G_{U^{(-\beta)}}(x)} \leq \frac{1}{\eta^{(\alpha,\beta)}|K+\epsilon|\E[T_2]Q_\beta(u_0)}.
        \end{align*}

        Finally, we consider the case that neither is subexponential. Since $Y$ is of strong subexponential type, the tails of $\tilde{\xi}_{\tilde{\zeta}_\beta}^{(\beta)}$ and $U^{(-\beta)}$ are dominated by the tails of at least one of either $\tilde{\xi}_{\tilde{\zeta}_{b}}^{(b)}$ or $U^{(-b)}$. Denote the dominating random variable by $X$ and let $W\in\{ \tilde{\xi}_{\tilde{\zeta}_\beta}^{(\beta)} , U^{(-\beta)} \}$. Following the above calculation, for all $\delta>0$ there exists an $R>0$ such that for any  $x>R$ and $n\in\N$,
        \begin{align*}
            \Prob_\alpha\left( W \geq x-(K+\epsilon)T_{2n}-A-u_0 \right)
            \leq \delta\Prob_\alpha\left( X>x-(K+\epsilon)T_{2n}-A-u_0 \right).
        \end{align*}
        Then, by using the results of the previous two cases, for suitably large $R>0$ we have
        \begin{align*}
            \sum_{n=0}^\infty \Prob_\alpha\left(W\geq x-(K+\epsilon)T_{2n}-A-u_0 \right)
            &\leq \frac{\delta(1+\delta)G_X(x)}{|K+\epsilon|\E[T_2]H_b(u_0)}
        \end{align*}
                \if{}
                and similarly
                \begin{align*}
                    \sum_{n=0}^\infty \Prob_\alpha\left(U_{-b}\geq x-(K+\epsilon)T_{2n}-A-u_0 \right)
                    &\leq \frac{\delta(1+\delta)G_X(x)}{|K+\epsilon|\E[T_2]Q_b(u_0)}.
                \end{align*}
                \fi
        Hence, for $x>R$,
        \begin{align*}
            \frac{\Prob\left( Z_n >x\:|\:K_{n-1}=\alpha,K_n=\beta \right)}{G_{\xi_{-b}}(x)+G_{U^{(b)}}(x)}
            &\leq \frac{2\delta(1+\delta)}{\eta^{(\alpha,\beta)}|K+\epsilon|\E[T_2]Q_b(u_0)}
        \end{align*}
        and so as $x\rightarrow\infty$, since $\delta>0$ was arbitrary
        \begin{align*}
            \Prob\left( Z^{(\alpha,\beta)}>x \right) = o\left( G_{\xi_{b}}(x)+G_{U^{(-b)}}(x) \right).
        \end{align*}

        Since all the components of the Lamperti-Kiu decomposition are finite we have, for sufficently large $x$, $G_{(\cdot)}(x) = H_{(\cdot)}(x)$. Hence, in the first two cases, we obtain
        \begin{align*}
            \limsup_{x\rightarrow\infty}\frac{\Prob(Z_n >x\:|\:K_{n-1}=\alpha,K_n=\beta )}{H_{\xi_\beta}(x) + H_{U^{(-\beta)}}(x)} \leq \frac{1}{\eta^{(\alpha,\beta)} Q(u_0)|K+\epsilon|\E[T_2]}.
        \end{align*}
        As we are taking the limit $x\rightarrow\infty$, we may also take $u_0\rightarrow\infty$ and use that $Q(u_0)\rightarrow1$ to obtain
        \begin{align*}
            \limsup_{x\rightarrow\infty}\frac{\Prob(Z_n >x\:|\:K_{n-1}=\alpha,K_n=\beta) }{H_{\xi_\beta}(x) + H_{U^{(-\beta)}}(x)} \leq \frac{1}{\eta^{(\alpha,\beta)}|K+\epsilon|\E[T_2]}
        \end{align*}
        whilst in the third case
        \begin{align*}
            \limsup_{x\rightarrow\infty}\frac{\Prob(Z_n >x\:|\:K_{n-1}=\alpha,K_n=\beta) }{H_{\xi_{-\beta}}(x) + H_{U^{(\beta)}}(x)} = 0.
        \end{align*}
        This completes the proof.
    \end{proof}

    \begin{lem}
        \label{lem:tailBound}
        Suppose $Z$ is a real-valued random variable with tail $\Prob(Z\geq x)$ which is bounded above by some function $F(x)$ such that $1-F(x)$ is a true distribution function. Then there exists a random variable $X$, which is a function of $Z$ and an independent uniformly distributed random variable, such that $Z\leq X$ and  $\Prob(X\geq x) = F(x)$.
    \end{lem}
    \begin{proof}
        Let $P(x) := \Prob(Z\geq x)$ and $V\sim \text{Unif}(0,1)$ be independent of $Z$. We will use the notation $P(x^+):=\lim_{y\downarrow x}P(y)$ which exists for all $x\geq0$ since $P$ is non-increasing and bounded from below. Define the random function $U:\R^+\rightarrow[0,1]$ by setting
        $
            U(x) := P(x) - V(P(x)-P(x^+)).
        $
        Let $x_1<x_2$, then since $P$ is non-increasing, $P(x^{+})\leq U(x)\leq P(x)$ for all $x\in\R$ and
        \begin{align*}
            U(x_2) \leq P(x_2) \leq P(x_1^+) = P(x_1) -1(P(x_1) - P(x_1^+))
            \leq U(x_1),
        \end{align*}
        hence $U$ is also non-increasing.

        Furthermore, suppose $U(x_1)=U(x_2)$ for some $x_1<x_2$. Then
        $
            P(x_1^+) \leq U(x_1) = U(x_2) \leq P(x_2) \leq P(x_1^+) ,
        $
        where the last inequality is because $P$ is non-increasing and so $P(x_1^+)=P(x_2)$. If $P(x_1^+)=P(x_1)$, then we have $P(x_2)=P(x_1)$. Otherwise, we have
        \begin{align*}
            U(x_1) > P(x_1^+) \geq P(x_2) > U(x_2).
        \end{align*}
        This is a contradiction. Hence, if $x_1<x_2$ and $U(x_1)=U(x_2)$ then $P(x_1^+)=P(x_1)=P(x_2)$ a.s. and so $\Prob( x_1 \leq Z < x_2) = 0$.

        From this, we can now conclude that, for all $x\in\R$, we have
        \begin{align*}
            \Prob( U(z) \leq U(x) )
            =\Prob(Z \geq x ) + \Prob( Z<x\:;\: U(Z)=U(x) )
            =P(x) + \Prob(Z<x\:;\: U(Z)=U(x)).
        \end{align*}
        However, by the above calculation, we get
        \begin{align*}
            \Prob(Z<x\:;\: U(Z)=U(x))
            \leq \Prob(Z<x\:;\: P(Z)=P(x) )
            =0
        \end{align*}
        and hence, for all $x\in\R$, we have
        $
            \Prob(U(z) \leq U(x) ) = P(x).
        $

        Now, let $q\in[0,1]$ and suppose that there exists $x\in\R^+$ with $P(x^+)=P(x)=q$ so
        \begin{align*}
            \Prob( U(Z) \leq q)
            &= \Prob( U(Z) \leq P(x) )
            =  \Prob( U(Z) \leq U(x) )
            =  P(x)
            =  q.
        \end{align*}
        If there is not such an $x$ then, since $\lim_{x\rightarrow-\infty}P(x)=1$ and $\lim_{x\rightarrow\infty}P(x)=0$, there exists $x\in\R^+$ such that $q\in[P(x^+),P(x))$. For any $y>x$, we have
        $
           U(y) \leq P(y) \leq P(x^+),
        $
        so $U(y)\notin(P(x^+),P(x))$ and similarly, for any $y<x$, we have
        $
            U(y)  \geq U(y^+) \geq P(x)    ,
        $
        so $U(y)\notin(P(x^+),P(x))$. Hence, $U(Z)\in(q,P(x))\subset(P(x^+),P(x))$ implies $Z=x$.

        Next, we want to consider $\Prob( P(x)=U(Z)\:;\: Z\neq x)$. Notice that if $z>x$ then $U(z)\leq P(x^+)<P(x)$ and so $\Prob( P(x)=U(z)\:;\: Z>x)=0$. If $z<x$ and $P(x)=U(z)$ then $P(x)=U(z)\geq P(z^+)\geq P(x)$ and so $P(x)=P(z^+)$. However, if there is a discontinuity point $y\in[z,x)$ then $U(z)\geq P(z^+)\geq P(y)>P(y^+)\geq P(x)$ so $\Prob(P(x)=U(z)\:;\: Z<x)=0$ we conclude $P(z)=P(z^+)=P(x)$ and hence,
        $
            \Prob( P(x)=U(Z) \:;\: Z\neq x ) = 0.
        $

        From this and since $V$ is uniformly distributed on $[0,1]$, we have
        \begin{align*}
            \Prob(U(Z)\in(q,P(x)])
            &= \Prob(Z=x)\Prob(U(x)\in(q,P(x)]) + \Prob(Z\neq x; U(Z)=P(x) )\\
            &= \Prob(Z=x)\Prob\left( \frac{q-P(x^+)}{P(x)-P(x^+)} < V \leq \frac{P(x)-P(x^+)}{P(x)-P(x^+)}\right) + 0\\
            &= \left(P(x)-P(x^+)\right)\frac{q-P(x^+)}{P(x)-P(x^+)},
        \end{align*}
        hence,
        \begin{align*}
             \Prob( U(Z) \leq q )
            &= \Prob( U(Z) \leq P(x) ) - \Prob( U(Z) \in (q,P(x)] )\\
            &= P(x) - (P(x) - P(x^+))\frac{(P(x)-q)}{P(x) - P(x^+)}
            =q.
        \end{align*}
        Hence, $U(Z)$ is uniformly distributed on $[0,1]$. Using \cite[pp 7, Proposition 3.1]{Embrechts2013}, the random variable defined by $X:= U^{-}(Z):=\inf\left\{ x\in\R^+ \:\middle|\: F(x) < U(Z) \right\}$ has the distribution $F$.
        Moreover, since $F(x)\geq P(x)$ for all $x\in\R^+$, we have
        \begin{align*}
            F(Z)
            \geq P(Z)
            \geq P(Z) - V(P(Z)-P(Z^+))
            = U(Z)
        \end{align*}
        and, since $F$ is non-increasing, we have
        $
            X = \inf\left\{ x\in\R^+ \:\middle|\: F(x) < U(Z) \right\} \geq Z,
        $
        as required.
    \end{proof}

    We are now able to derive the upper bound of Theorem \ref{thm:subexponTails}.
    \begin{lem}
        \label{lem:expTailsUpper}
        If $Y$ is a Lamperti-Kiu process of strong subexponential type, then the right tail of $A_\infty$ satisfies
        \begin{align}
            \label{eqn:expTailsUpper}
            \limsup_{x\rightarrow\infty}\frac{\Prob(\log(A_\infty)>x)}{H(x)} \leq \frac{1}{|K|\E[T_2]},
        \end{align}
        where $H$ is the integrated tail from Theorem \ref{thm:subexponTails}.
    \end{lem}

    \begin{proof}
        Fix $\sigma\in\{+,-\}$ and let $\delta_2>0$. For sufficiently large $A>0$, by Proposition \ref{prop:markovChainK}, we know \newline $\E_\sigma[N(\alpha,\beta)]/\eta^{(\alpha,\beta)}\leq \mathbbm{1}_{\{\sigma=\alpha\}} + \delta_2$. Now fix such an $A>0$ and let $\delta_1>0$.

        From Lemma \ref{lem:logAsum}, we have
        $
            \log(A_\infty)
            \leq (N+1)C + \sum_{i=1}^NZ_i^+.
        $

        For each $i\in\N$, we have a tail estimate for $Z_i^+$, given $(K_n)_{n\in\N_0}$, from Lemma \ref{lem:limZalphaBeta} which, used in conjunction with Lemma \ref{lem:tailBound}, gives the existence of random variables $X_i(k)$  for each $k\in\cup_{n\in\N}\{+,-\}^n$ with $k_0=\sigma$ such that:
        \begin{enumerate}
            \item each $X_i(k)$ is a function of $Z_i$ and a random variable independent of the rest of the system;
            \item $X_i(k)\geq (Z_i^+ +C)\mathbbm{1}_{\{N=n;(K_0,\cdots,K_n)=k\}}$;
            \item $X_i(k)$ has tails given by $\min\left(1,H^{(k_i)}(x)(\eta^{(k_{i-1},k_i)}|K+\epsilon|\E[T_2]|)^{-1}\right)$ if $k_i\in B$ and tails which are $o(\min(1,H^{(b)}(x)))$ for $b\in B$ if $k_i\notin B$.
        \end{enumerate}
        Then, summing up over the sample paths of $(K_n)_{n\in\N_0}$, we have the upper bound
        \begin{align*}
            \Prob_\sigma\left( \sum_{i=1}^N Z_i^+ +C>x \right)
            &\leq\sum_{n\in\N}\sum_{\substack{ k\in\{+,-\}^{n+1} \\ k_0=\sigma } }\Prob_\sigma\left( \sum_{i=1}^n X_i(k)>x; N=n; (K_0,\cdots,K_n)=k \right).
        \end{align*}
        For ease of notation, let $\bar{\eta}=\max_{\alpha,\beta\in\{+,-\}}\eta_{\alpha,\beta}$, $\underline{\eta}=\min_{\alpha,\beta\in\{+,-\}}\eta_{\alpha,\beta}$ and $d\in(0,(1-2\bar{\eta})/2\bar{\eta})$.

        For each $\alpha,\beta\in\{+,-\}$, $n\in\N$ and $k\in\{+,-\}^{n+1}$ such that $k_0=\sigma$ let $n_{\alpha,\beta}(k):=\sum_{i=0}^n\mathbbm{1}_{\{k_{i-1}=\alpha,k_i=\beta\}}$. Then, by Lemma \ref{lem:xiIndp}, given the event $\{N=n;(K_0,\dots,K_n)=k\}$, the sum $Y_{\alpha,\beta}(k):=\sum_{i=1}^nX_i\mathbbm{1}_{\{k_{i-1}=\alpha,k_i=\beta\}}$ is a sum of $n_{\alpha,\beta}(k)$ i.i.d. random variables hence, in the case $\beta\in B$, from Kesten's bound \cite[pp 67, Section 3.10, Theorem 3.34]{foss2013introduction},
        \begin{multline*}
            \Prob_\sigma\left( \sum_{i=1}^nX_i(k)\mathbbm{1}_{\{k_{i-1}=\alpha,k_i=\beta\}} >x\:\middle|\:N=n;\:(K_0,\dots,K_n)=k  \right)\\
            \leq c(d)(1+d)^{n_{\alpha,\beta}(k)}\Prob_\sigma(X_1((\alpha,\beta))>x)
            \leq \frac{c(d)(1+d)^nF(x)}{\underline{\eta}|K+\epsilon|\E[T_2]}.
        \end{multline*}
        In the case $\beta\notin B$, since $Y$ is of strong subexponential type, for any $b\in B$,
        \begin{multline*}
            \Prob_\sigma\left( \sum_{i=1}^nX_i(k)\mathbbm{1}_{\{k_{i-1}=\alpha,k_i=\beta\}} >x\:\middle|\:N=n;\:(K_0,\dots,K_n)=k  \right)\\
            \leq \Prob_\sigma\left( \sum_{i=1}^nW_i\mathbbm{1}_{\{k_{i-1}=\alpha,k_i=\beta\}} >x\:\middle|\:N=n;\:(K_0,\dots,K_n)=k  \right)
            \leq \frac{c(d)(1+d)^nH(x)}{\underline{\eta}|K+\epsilon|\E[T_2]},
        \end{multline*}
        where for each $i\in\N$ the random variable $W_i$ depends only on $X_i(k)$ and has the distribution of $X_1((\alpha,b))$. We can now use corollaries 3.16 and 3.18 in \cite[pp 52, Chapter 3]{foss2013introduction} to sum the $Y_{\alpha,\beta}$ and obtain the bound
        \begin{align*}
            &\Prob_\sigma\left( \sum_{i=1}^nX_i(k) >x\:\middle|\: N=n;(K_0,\cdots,K_n)=k \right)\\
            &\qquad=\Prob_\sigma\left( \sum_{\alpha,\beta\in\{+,-\}}\sum_{i=1}^nX_i(k)\mathbbm{1}_{\{k_{i-1}=\alpha,k_i=\beta\}} >x \:\middle|\: N=n;(K_0,\cdots,K_n)=k \right)\\
            &\qquad\leq \frac{4c(d)(1+d)^nH(x)}{\underline{\eta}|K+\epsilon|\E[T_2]}.
        \end{align*}

        Using the bound on the distribution of $N$ from equation (\ref{eqn:NGeomBound}), we see that there is an $M\in\N$ such that $\E[\underline{\eta}^{-1}4c(d)(1+d)^N;N>M]\leq \delta_1$ for sufficiently small $d$. This then gives
        \begin{align*}
            \Prob_\sigma\left( \sum_{i=1}^N X_i((K_n)_{n\in\N})>x; N>M \right) \leq \frac{\delta_1 H(x)}{|K+\epsilon|\E[T_2]}.
        \end{align*}

        Moreover, by \cite[pp 52, Chapter 3, Corollary 3.16]{foss2013introduction}, for all $n\leq M$ and $k\in\{+,-\}^{n+1}$ with $k_0=\sigma$ there exists $R_{n,k}>0$ such that for all $x>R_{n,k}$
        \begin{multline*}
            \Prob_\sigma\left( \sum_{i=1}^n X_i(k)>x\:\middle|\: N=n; (K_0,\cdots,K_n)=k \right)
            \\ \leq \left( \left(1+\frac{\delta_1}{2}\right)\sum_{\alpha\in\{+,-\}}\sum_{\beta\in B} n_{\alpha,\beta}(k)\frac{H^{\beta}(x)}{\eta^{(\alpha,\beta)}|K+\epsilon|\E[T_2]}\right)
            \bar{*} \Prob_\sigma\left( \sum_{i=1}^n \mathbbm{1}_{\{k_i \notin B\}} X_i(k) > x \right)(x),
        \end{multline*}
        where for two survival functions $\bar{F}$ and $\bar{G}$ we write $\bar{F}\bar{*}\bar{G}$ for the survival function of the convolution of the distribution functions $F \coloneqq 1-\bar{F}$ and $G\coloneqq 1-\bar{G}$. Then, since $\Prob_\sigma\left( \mathbbm{1}_{\{k_i\notin B\}}X_i(k))>x \right) = o\left( H^{\beta}(x) \right)$ for any $\beta\in B$, by \cite[pp 52, Chapter 3, Corollary 3.18]{foss2013introduction} we have
        \begin{align*}
            \Prob_\sigma\left( \sum_{i=1}^n X_i(k)>x\:\middle|\: N=n; (K_0,\cdots,K_n)=k \right)
             \leq \left(1+\delta_1\right)\sum_{\alpha\in\{+,-\}}\sum_{\beta\in B} n_{\alpha,\beta}(k)\frac{H^{\beta}(x)}{\eta^{(\alpha,\beta)}|K+\epsilon|\E[T_2]}.
        \end{align*}

        Since there are finitely many such pairs $(n,k)$ with $n\leq M$ we can take $R:=\max_{n\leq M}R_{n,k}$ so that for all $x>R$, we get
        \begin{align*}
            &\sum_{n=1}^{M}\sum_{ \substack{ k\in\{+,-\}^{n+1} \\ k_0=\sigma }} \Prob_\sigma\left( \sum_{i=1}^nX_i(k)>x ; N=n; (K_0,\cdots,K_n)=k \right)\\
            &\leq (1+\delta_1)\sum_{n=1}^{M}\sum_{\substack{ k\in\{+,-\}^{n+1} \\ k_0=\sigma }}\sum_{\alpha\in\{+,-\}}\sum_{\beta\in B}n_{\alpha,\beta}(k)\frac{H^{(\beta)}(x)}{\eta^{(\alpha,\beta)}|K+\epsilon|\E[T_2]}\Prob_\sigma(N=n;(K_0,\cdots,K_n)=k)  \\
            &= (1+\delta_1)\sum_{\alpha\in\{+,-\}}\sum_{\beta\in B}\E_\sigma[n_{\alpha,\beta}(k);N\leq M]\frac{H^{(\beta)}(x)}{\eta^{(\alpha,\beta)}|K+\epsilon|\E[T_2]}\\
            &\leq \frac{(1+\delta_1)}{|K+\epsilon|\E[T_2]}\sum_{\alpha\in\{+,-\}}\sum_{\beta\in B}\frac{\E_\sigma[n_{\alpha,\beta}(k)]H^{(\beta)}(x)}{\eta^{(\alpha,\beta)}}.
        \end{align*}
        Next, by our choice of $A$ we have $\E_\sigma[n(\alpha,\beta)]/\eta^{(\alpha,\beta)}\leq \mathbbm{1}_{\sigma=\alpha} + \delta_2$ thus
        \begin{align*}
            &\sum_{n=1}^{M}\sum_{\substack{ k\in\{+,-\}^{n+1} \\ k_0=\sigma }} \Prob_\sigma\left( \sum_{i=1}^nX_i(k)>x ; N=n; (K_0,\cdots,K_n)=k \right)\\
            &\qquad\leq \frac{(1+\delta_1)}{|K+\epsilon|\E[T_2]}\sum_{\alpha\in\{+,-\}}\sum_{\beta\in B}\left( \mathbbm{1}_{\{\sigma=\alpha\}} + \delta_2 \right)H^{(\beta)}(x)\\
            &\qquad= \frac{(1+\delta_1)}{|K+\epsilon|\E[T_2]}(1+2\delta_2)\sum_{\beta\in B}H^{(\beta)}(x)\\
            &\qquad\leq\frac{(1+\delta_1)^2}{|K+\epsilon|\E[T_2]}(1+2\delta_2)H(x),
        \end{align*}
        where the last inequality holds for sufficiently large $x$ since  $Y$ is of strong subexponential type. Hence, for all $x>R$, we have
        \begin{align*}
            \Prob_\sigma\left( \sum_{i=1}^N \left( Z_i^+ +C \right) >x \right)
            \leq \frac{(1+\delta_1)^2(1+2\delta_2)}{|K+\epsilon|\E[T_2]}F(x) +\frac{\delta_1 H(x)}{|K+\epsilon|\E[T_2]},
        \end{align*}
        and so from the definition of $\limsup$,
        \begin{align*}
            \limsup_{x\rightarrow\infty}\frac{\Prob_\sigma\left( C + \sum_{i=1}^N \left(Z_i^++C\right)>x \right)}{H(x-C)}
            \leq \frac{1+2\delta_2}{|K+\epsilon|\E[T_2]}.
        \end{align*}
        However, because $H$ is long tailed $\lim_{x\rightarrow\infty}H(x-C)/H(x) =1$ and therefore,
        \begin{align*}
            \limsup_{x\rightarrow\infty}\frac{\Prob_\sigma\left( C + \sum_{i=1}^N \left(Z_i^++C\right)>x \right)}{H(x)}
            \leq \frac{1+2\delta_2}{|K+\epsilon|\E[T_2]}.
        \end{align*}
        Then, by comparison
        \begin{align*}
            \limsup_{x\rightarrow\infty} \frac{\Prob_\sigma(\log(A_\infty)>x)}{H(x)}
            &\leq \frac{1+2\delta_2}{|K+\epsilon|\E[T_2]}
        \end{align*}
        and since both $\epsilon$ and $\delta_2$ were arbitrary the result follows.
\end{proof}

    It remains to show the lower bound for $\liminf_{x\rightarrow\infty}\Prob(A_\infty>x)$ also holds. To this end, for each $x\in\R$, define the stopping time
    \begin{equation}
        \tau_d(x) := \inf\left\{ T_{2n} \:|\:n\in\N, \xi_{T_{2n}} \geq x \right\}
    \end{equation}
    and notice that $\tau_d(x)<\infty$ if and only if $\sup_{n\in\N}\xi_{T_{2n}}\geq x$. Furthermore, $J_{\tau_d(x)}=J_0$ whenever $\tau_d(x)<\infty$. Since $\xi_{T_2}$ is strong subexponential its integrated tail, $H$, is subexponential  and thus, by \cite[pp 2, Theorem 1(ii)]{ZacharyVeraverbeke}, $\Prob(\tau_d(x)<\infty)$ is also subexponential. Then, by considering the random walk $\left(\xi_{T_{2n}}\right)_{n\in\N}$ in the place of $\left( \xi_n \right)_{n\in\N}$ in the proof of \cite[pp 12, Lemma 4.3]{MaulikZwart2006156}, we have that for every $y>0$,
    \begin{equation}
        \label{eqn:txitaudLongTailed}
        \lim_{x\rightarrow\infty} \Prob\left(\xi_{\tau_d(x)}-x>y \:\middle|\: \tau_d(x)<\infty \right) = 1.
    \end{equation}

    We are now able to prove Theorem \ref{thm:subexponTails}.

    \begin{proof}[Proof of Theorem \ref{thm:subexponTails}]
        Equation (\ref{eqn:expTails}) of Theorem \ref{thm:subexponTails} follows from the inequality (\ref{eqn:expTailsUpper}) of Lemma \ref{lem:expTailsUpper} and the inequality
        \begin{equation}
                \liminf_{x\rightarrow\infty} \frac{\Prob\left(\log\left(A_\infty\right) > x\right)}{H(x)} \geq \frac{1}{|\E[\xi_{T_2}]|} = \frac{1}{\E[T_2]|K|},
        \end{equation}
        which we will now prove.

        The following inequality is immediate
        \begin{align*}
            \Prob(\log A_\infty> x)
            &\geq \Prob\left( \log\left( \int_{\tau_d(x)}^\infty |Y_t| dt \right) > x; \tau_d(x)<\infty \right).
        \end{align*}
        By applying the Markov additive property and recalling $J_{\tau_d(x)}=J_0$ gives
        \begin{align*}
            \Prob(\log A_\infty> x)
            &\geq\Prob\left( \xi_{\tau_d(x)}  + \log(\hat{A}_{\infty,J_0}) >x \:\middle|\: \tau_d(x)<\infty \right)\Prob(\tau_d(x)<\infty),
        \end{align*}
        where $\hat{A}_{\infty,j}$ is an independent and identically distributed copy of $A_\infty$ with $\hat{J}_0=j$. Then by applying equation (\ref{eqn:txitaudLongTailed}), we have
        \begin{align*}
            \liminf_{x\rightarrow\infty} \frac{\Prob\left( \log A_{\infty} > x \right)}{H(x)}
            &\geq\liminf_{x\rightarrow\infty} \frac{\Prob(\tau_d(x)<\infty )}{H(x)}
            =\liminf_{x\rightarrow\infty} \frac{\Prob(\sup_{n\in\N}\xi(T_{2n})\geq x)}{H(x)}.
        \end{align*}
        However, $\xi_{T_{2n}}$ is a sum of the random variables $\xi_{T_{2m}}-\xi_{T_{2(m-1)}}$ which are i.i.d. copies of $\xi_{T_2}$. Since $Y$ is of strong subexponential type, the integrated tail, $H$, of $\xi_{T_2}$ is long-tailed and we can apply Veraverbeke's theorem \cite[pp 2, Theorem 1(i)]{ZacharyVeraverbeke} to conclude that
        \begin{align*}
            \liminf_{x\rightarrow\infty}\frac{\Prob\left(\sup_{n\in\N}\xi(T_{2n}) \geq x\right)}{H(x)} \geq \frac{1}{|\E[\xi_{T_2}]|}
            = \frac{1}{|K|\E[T_2]}.
        \end{align*}
    \end{proof}

\begin{rmk}
 The results of this paper are presented only for settings where $|E| = 2$. However, they can easily be extended to the case of any finite $E$ provided the modulating Markov chain $J$ is irreducible. In the proofs, extensive use is made of the fact that $J_{T_{2n}} = J_0$ for all $n\in N$. To extend this to the case $|E| > 2$, we replace $\{T_{2n}\}_{n\in N }$ with the sequence of return times to $J_0$ of $J$, which have finite expectation. In the case $|E| =\infty$, two further difficulties arise which may prevent an extension of the results. Firstly, even if $J$ is a recurrent Markov chain, it may be the case that the expected return time of $J$ is infinite. Secondly, arguments which rely on taking  maximums or sums over the elements of $E$ are no longer valid. The reader may also be interested by the work in \cite{Beheme-sideris} where a necessary and sufficient condition for finiteness of  $\int_0^{\infty} e^{\xi_s}d\eta_s $, where $(\xi, \eta)$ is a bivariate Markov additive process with some modulating Markov chain $J$, is  given.
\end{rmk}
{\bf Acknowledgements:} We are grateful to the anonymous referees for their careful reading of our manuscript and their valuable comments which helped improving the paper.

  \end{document}